\documentclass[preprint,12pt]{elsarticle}

\usepackage[margin=2.0cm]{geometry}

\usepackage{graphicx}
\usepackage{subcaption}
\usepackage{amssymb}
\usepackage{amsmath}
\usepackage{siunitx}
\usepackage[dvipsnames]{xcolor}
\usepackage{commath}
\usepackage{amsfonts}%
\usepackage{graphicx,verbatim}
\usepackage{tikz}
\usetikzlibrary{cd}

\usepackage{float}

\usepackage[pagebackref=false,colorlinks=true, pdfstartview=FitV, linkcolor=violet,citecolor=red, urlcolor=blue]{hyperref}

\usepackage{amsthm}
\newtheorem{lthm}{Theorem}

\newtheorem{ldef}{Definition}

\usepackage{dirtytalk}

\usepackage{hyperref}

\newcommand{\cC}{\mathcal{C}}
\newcommand{\cT}{\mathcal{T}}

\usepackage[font=small,labelfont=bf]{caption}

\newtheorem{theorem}{Theorem}
\newtheorem*{theorem*}{Theorem}
\newtheorem{definition}{Definition}
\newtheorem{proposition}{Proposition}
\newtheorem{corollary}{Corollary}
\newtheorem{lemma}{Lemma}
\newtheorem{example}{Example}

\newtheorem*{conjecture*}{Conjecture}

\newtheorem{question}{Question}
\newtheorem{remark}{Remark}
\usepackage{lineno}

\journal{ArXiv}

\begin{document}

\begin{frontmatter}


\title{Knots and the Sierpinski tetrahedron}



\author{Malors Espinosa}
\ead{espino41@mail.tsinghua.edu.cn, Yau Mathematical Science Center, Tsinghua University}


\begin{abstract}

In this paper we prove that there are infinitely many knots that cannot be embedded in the $1$-skeletons of the \textit{finite} iterations of the Sierpinski tetrahedron fractal. We do this by proving that such an embedding induces a sphere decomposition of weight at most $6$. There are infinitely many knots with spherewidth greater than this.

\end{abstract}
\end{frontmatter}


\section{Introduction}\label{sec: introduction}

\subsection{Motivation and main results}

The Sierpinski tetrahedron is a fractal constructed out of an iterative process. Concretely, a regular solid tetrahedron gets substituted by its homothetic copies, with proportion of 1/2, at each one of its vertices. At each step, this substitution is carried out on each remaining solid tetrahedrons. The intersection of this sequence of nested closed sets is the fractal.

\begin{figure}[H]
    \centering
    \includegraphics[width=0.4\linewidth]{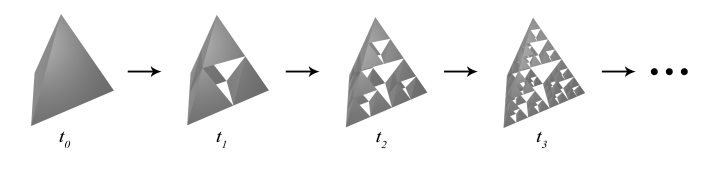}
    \caption{The different iteration stages of the construction}
   \label{fig:SierpinskiTetrahedronProgression}
\end{figure}

The 1-skeletons of the finite stages is a particular sequence of graphs, all of which are subsets of the Sierpinski tetrahedron. We see from this that this fractal has embedded many polygonal curves. 

In \cite{knotsinsidefractals} we looked for some of these polygonal curves that were knotted. In other words, we wondered whether all (tame) knots can be found as subsets of the Sierpinski tetrahedron. Our strategy was to find the knots as subpaths of some big enough 1-skeleton. This then implies it is also in the fractal itself.

We were able to prove that all Pretzel knots can be found in the 1-skeletons. We conjectured that, in fact, all knots should be there in the final fractal. The purpose of this paper is to prove the following result.

\begin{lthm}
\label{lthm: not all knots}
Let $T(p, q)$ be a torus knot with $\min\{p, q\} > 9$, then $T(p, q)$ cannot be embedded in the $1$-skeleton of a finite iteration of the Sierpinski tetrahedron. In particular, there are infinitely many knots that cannot be embedded in the $1$-skeletons of the finite iterations of the Sierpinski tetrahedron.
\end{lthm}

This theorem does not answer completely whether all knots are in the Sierpinski tetrahedron, but it certainly casts doubts on the veracity of the claim. When we looked for Pretzel knots, for the results of \cite{knotsinsidefractals}, we did it by a combinatorial search. There didn't seem to be any major reason why a more patient and methodical search would not be able to lead to any given knot. Now we know that, if some version of the conjecture is true, there will need to be tame knots that indeed use the fractal structure of the Sierpinski tetrahedron to avoid the restrictions the finite iterations pose. 

However, in the quest of understanding tame knots that we find with wild presentations, the fractal is very misleading and at the moment I don't know if the results of this work generalize in some way or if there are genuinely new knots to be found in the Sierpinski tetrahedron but not a moment before. 

The main constructions that serves our proof are \textit{sphere decompositions} and the \textit{spherewidth} of a knot. Following \cite{LunelSphereWidth}, we have the following two definitions.
\begin{ldef}
 A \textbf{thickened embedding} of a graph $G$ is an embedding of $G$ in $\mathbb{S}^3$ where each vertex is thickened to a small ball. Two balls are connected by a polygonal edge if an only if they are adjacent in the graph $G$, and pairs of edges are disjoint.   
\end{ldef}

\begin{ldef}
    \label{ldef: sphere decomposition}
Let $G$ be a graph thickly embedded in $\mathbb{S}^3$. A \textbf{sphere decomposition} of $G$ is a continuous map $f:\mathbb{S}^3\longrightarrow \cT$ where $\cT$ is a trivalent tree with at least one edge. It satisfies
\begin{description}
    \item[(SD 1)] For all $x\in L(\cT)$, $f^{-1}(x)$ is a point disjoint from $G$.

    \item[(SD 2)] For all $x\in V(\cT)-L(\cT)$ , $f^{-1}(x)$ is a double bubble transverse to $G$.

    \item[(SD 3)] For all $x\in \cT$ interior to an edge, $f^{-1}(x)$ is a sphere transverse or finitely tangent to $G$.
\end{description}
\end{ldef}
You can see examples of double bubbles in figures \ref{fig: generic double bubble} and \ref{fig:doublebubbles}. We then have
\begin{ldef}
    \label{ldef: width}
Let $G$ be a graph and $f:\mathbb{S}^3\longrightarrow \cT$ be a sphere decomposition of $G$. We define \textbf{the width of $f$} as the supremum of the weights of the spheres $f^{-1}(x)$ for all $x$ interior to the edges. That is,
\begin{equation*}
    w(f) := \displaystyle\sup_{x\in G - V(G)} w(f^{-1}(x)).
\end{equation*}
Furthermore, the \textbf{spherewidth} of $G$ is the infimum of the widths of its sphere decompositions. That is,
\begin{equation*}
    sw(G) := \displaystyle\inf_{f:\mathbb{S}^3\longrightarrow \cT} w(f).
\end{equation*}
\end{ldef}
In the above definition, the weight of a sphere $S$ with respect to $G$ is the number of connected components of $G\cap S$. 

In vague terms, a sphere decomposition is a \say{movie} that shows the growth of several spheres. They start as points, they grow and merge, and eventually they become a single sphere that grows towards infinity. This movie is not parametrized by time, but by the trivalent tree $\cT$. If you were to stop at any point of the interior of the edges, you would see a snapshot of how a sphere is growing. In the process of growth, these spheres hold in their interior part of the given knot until they hold it entirely. 

The knot is a single component, however, as the spheres grow they might collect different parts of the knot that are disconnected and will only merge at a later point. The spherewidth $sw(K)$ of a knot $K$, then, is an unavoidable quality of the sphere growth: no matter how you run this process, somewhere along the way, at least one of the spheres intersected the knot at least a certain fixed amount of times.

Consequently, if one is able to produce a sphere decomposition where at every point, interior to the edges of $\cT$, the number of components, in the intersections with the knot, is bounded above by an absolute constant, then the spherewidth is bounded above by this same constant.

We prove the following result.

\begin{lthm}
    \label{lthm: sphere decomposition}
Let $K$ be a knot that can be embedded in a finite iteration of the Sierpinski tetrahedron. Then $K$ admits a sphere decomposition of weight at most $6$.
\end{lthm}

Theorem \ref{lthm: not all knots} follows from theorem \ref{lthm: sphere decomposition} because it is known (see \cite{LunelSphereWidth}) that
\begin{equation*}
    sw(T(p, q))\ge \dfrac{2}{3}\min\{p, q\}.
\end{equation*}

Theorem \ref{lthm: sphere decomposition} follows by constructing an specific sphere decomposition from a finite iteration and the way the knot embeds in its $1$-skeleton. Intuitively, this sphere decomposition is very clear.  We see \say{snapshots} of this growth when we pick an iteration, from the zeroth one up to the one that includes the knot, and take as bubbles the tetrahedrons themselves. The challenge lies on proving that these snapshots are indeed part of a larger picture (i.e. one can build the sphere decomposition out of this data). Finally, the reason why the absolute bound is achieved is that every tetrahedron can contain at most two connected components of the knot (because there are only four vertices!). Thus, no matter how one tetrahedron grows into a bigger one, the components must glue together into at most two components. This constrain limits severely how the number of components can grow within a single sphere. 

\subsection{Organization of this paper}

In our paper we will work with polyhedrons. In section \ref{sec: Some geometric observations on polyhedra} we will review some properties and definitions we need from them. This section is complemented in the \ref{appx: Geometric Lemmas on Polyhedrons} where one of our main construction results is proven. We put this in an appendix to do not break the flow of the argument, having in mind that what is required there is not used, beyond the result itself, for the rest of the paper.

The choice of working on polyhedrons was taken by the author after trying to do all this somewhat explicitly within the smooth category. Nothing was really problematic except to guarantee that, at every stage that is needed, one can satisfy the conditions of transversality or finite tangency. I could not find a result that properly stated that all the constructions I was making did not, at some unfortunate point, created problems due to how it altered the knot. I am sure most of this constitutes nothing else than my own limitations on the knowledge of the area. On top of this, A.I. was screaming at me that this was absolutely true and known to everyone, except obviously me. This gave me pause and I decided to simply work with polyhedrons. I apologize if I overcomplicated the task without need.

In section \ref{sec: Sphere decomposition and spherewidth} we review the concepts of sphere decompositions and of spherewidth. We follow the work of \cite{LunelSphereWidth} for this. In section \ref{sec: Sphere Decomposition induced by the Sierpinski tetrahedron} we prove theorem \ref{lthm: not all knots} and theorem \ref{lthm: sphere decomposition}. This is a long section. We begin by modifying the construction of the finite iterations with tetrahedrons with a different one which, eventually, will let us to work with cubes. Then we will explain how to obtain the combinatorial information necessary to construct the sphere decomposition. We do this carefully to be able to control the growth of components within intermediate stages of the growth process. In here, we put special attention to the particular thickening of the vertices. We do this to make sure they do not make our bounds grow because we chose a bad construction.

In section \ref{sec: An example with the Trefoil} we make a concrete example with the trefoil. Our objective is to show that from the data of the embedding into a finite iteration one can indeed very easily read all the information needed to compute the sphere decompositon. 

Finally, in \ref{sec: Further Questions and Concluding Remarks} we pose two questions that follow from this work and that the author finds interesting. 

\subsection{Discussion on some previous literature and some comments}

During this work we will be assuming that the reader is familiar with the required contents of \cite{knotsinsidefractals}. The authors of the aforementioned paper, at the time of writing it, could not find a lot of material relating knots and fractals in the way we were interested in. Since then, some related results have appeared and some new information has been made available to the author.

To begin with, the work of Shtan'ko, \cite{shtanko1971menger}, settles the existence of knots already embedded in the Menger Sponge. However, it does it in a nonconstructive way and does not deal with the skeletons of the finite iterations, as we did in \cite{knotsinsidefractals}. The author was lead to this bibliography by none other than ChatGPT as we will discuss in our A.I. disclosure. Besides tame knots, the work of Shtan'ko also deals with wild knots, but once more does not do the embedding constructively. The recent work of Hinojosa, Morales-Fuentes, Valdez and Verjovsky, \cite{hinojosaetal}, deals with some of them in a constructive way. 

As far as my understanding goes, the results of Shtan'ko do not extend to the Sierpinski tetrahedron and thus do not complete our quest to know what happens in the fractal itself. Furthermore, I could not find (even with help of A.I.) literature that would immediately imply this result was already established for some unknown reason to me.

Related to our main tool, \textit{sphere decompositions}, the main work we use is Lunel's and De Mesmay's, \cite{LunelSphereWidth}. The idea, as we have mentioned before, is to understand how a knot can be \say{isolated} in growing spheres. However, similar ideas can be applied to knot diagrams. In particular, the concept of \textit{tree decompositions}, as developed in the work of De Mesmay, Purcell, Schleimer and Sedgwick, \cite{SchleimerTreeWidth}, is an approach the author tried but could not make work. The difference is that, when dealing with the knot diagrams that come from the combinatorial representations (see figure \ref{fig: combinatorial representation}), there are intersections that are really not there. They are simply the price to pay for flattening a highly non planar object. Thus, when trying to build \textit{bags}, as in \cite[Definition 2.1]{SchleimerTreeWidth}, with controlled size, these crossings made my bounds to not be uniform. This problem is what motivated me to try directly with the \textit{3D version} of the problem, which at first I tried to avoid.

For some nice examples of what happens when one deals with wild knots versus tame knots, I found Kobayashi's, \cite{kobayashi2021}, very illuminating.

\subsection{A.I. use disclosure}

In the process of writing this paper I have used ChatGPT and Deepseek. I am writing this introduction on the same week that the Navier-Stokes controversy is developing and I find it's best to be explicit with what we did and know with A.I.

ChatGPT Pro, the 20 dollar version, is what I used. I used it twice. The first time was to discuss the original result I wanted to prove. I discussed my main ideas on how one can possibly prove all knots were in the finite iterations. Nothing of depth ocurred. We discussed also the result of the Menger Sponge proven in \cite{knotsinsidefractals} with the hopes of adapting it to the Sierpinski. It is here where ChatGPT showed me Shtan'ko's work \cite{shtanko1971menger}. I wonder if it told me this because it was mentioned in the introduction of \cite{hinojosaetal}. I became aware of this latter paper by other means.

I eventually confessed to ChatGPT that recently I have doubted the conjecture to be true. I quote myself: \say{\textit{I do not know enough of knot theory to say this precisely, but I can imagine that there are knots that are REALLY knotted and thus any knot representative cramps no matter what you do. Hence, the connectivity of the Sierpinski simply does not give enough.}} (Yes, I used \textit{cramps} and I meant \textit{piles up}. I cannot explain.) At this comment, ChatGPT thought for a bit and referred to me to the concepts of \textit{spherewidth}, \textit{sphere decompositions} and \textit{tree decompositions}, as well as the papers \cite{LunelSphereWidth}, \cite{SchleimerTreeWidth}. Never before I have heard any of this.

It mentioned that my \say{cramped} knots would be torus knots and showed me the known lower bounds. We discussed for a long while because I could not grasp in its entirety all that it was explaining to me, but ultimately it turned into a plan of attack that became the backbone of what I present in this paper. I decided not to use more of ChatGPT for this project because I wanted to solidify the rest myself.

The second time I used ChatGPT was to know suggestions of programs I could use to make many of the several pictures you will see in this paper. It suggested many, among them Asymptote, which became extremely valuable. I did this because Deepseek did not suggest any option that seemed a good idea.

For the rest of the process, I worked with the free version of Deepseek. I used it to generate the codes to create the images in Asymptote. It helped me to do the tikz diagrams. It helped me to verify certain computational details in the propositions of the \ref{appx: Geometric Lemmas on Polyhedrons}. It also told me several results related to transversality in sphere decompositions are well known and provided me several sources. However, I did check this bibliography and I could not see in an straightforward way what I wanted. 

I have found working and learning with A.I. fascinating. It is deeply discouraging and somewhat painful to see what it is becoming.

\subsection{Ackwonledgements}

I take the opportunity to thank the mathematical community, whether researchers, educators or students, for their unexpected support in our work on knots and fractals. 

\section{Some geometric observations on polyhedra}
\label{sec: Some geometric observations on polyhedra}

Throughout the paper, including \ref{appx: Geometric Lemmas on Polyhedrons}, we will prominently refer to two families of polyhedra. We will agree that their definition is understood by the following drawings together with some brief explanations about them. All of these figures are well known.

\begin{description}
    \item[A rectangular prism:] A rectangular prism is a solid with six rectangular faces and any two faces that share an edge are orthogonal.

    \item[A prismoid or loft:] A \textit{prismoid or loft} is a convex polyhedron with two rectangular faces $ABCD$ and $abcd$ and corresponding vertices ($A$ with $a$, $B$ with $b$, etc) joined by an edge. 
\end{description}

We can see an example of each of the above polyhedra in figure \ref{fig: all polyhedrons}.

\begin{figure}[H]
    \centering
    \begin{subfigure}{0.48\textwidth}
        \centering
        \includegraphics[scale = 0.1]{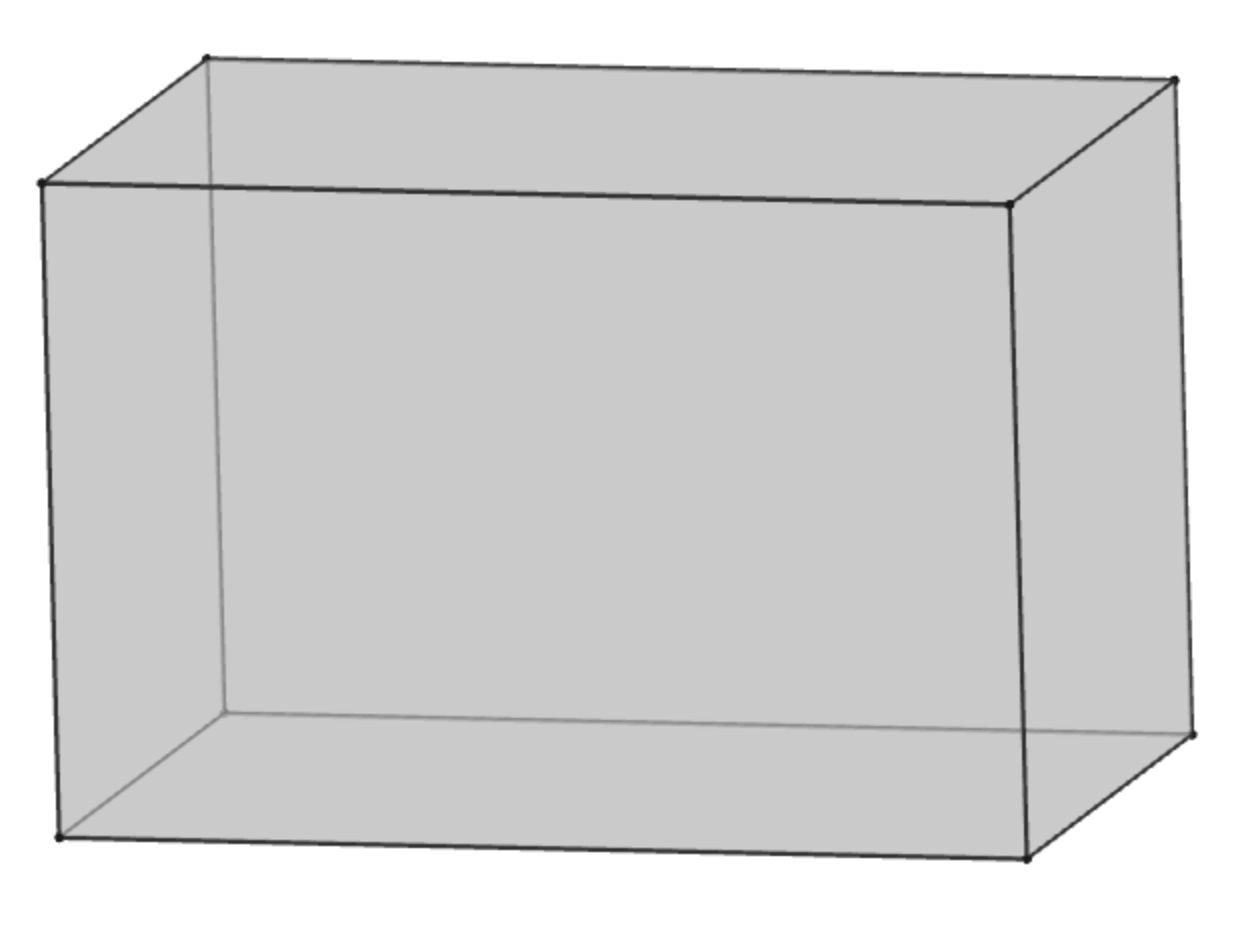}
        \caption{Rectangular Prism}
        \label{fig: rectangular prism}
    \end{subfigure}
    \hfill 
    \hfill 
    \begin{subfigure}{0.48\textwidth}
        \centering
        \includegraphics[scale = 0.1]{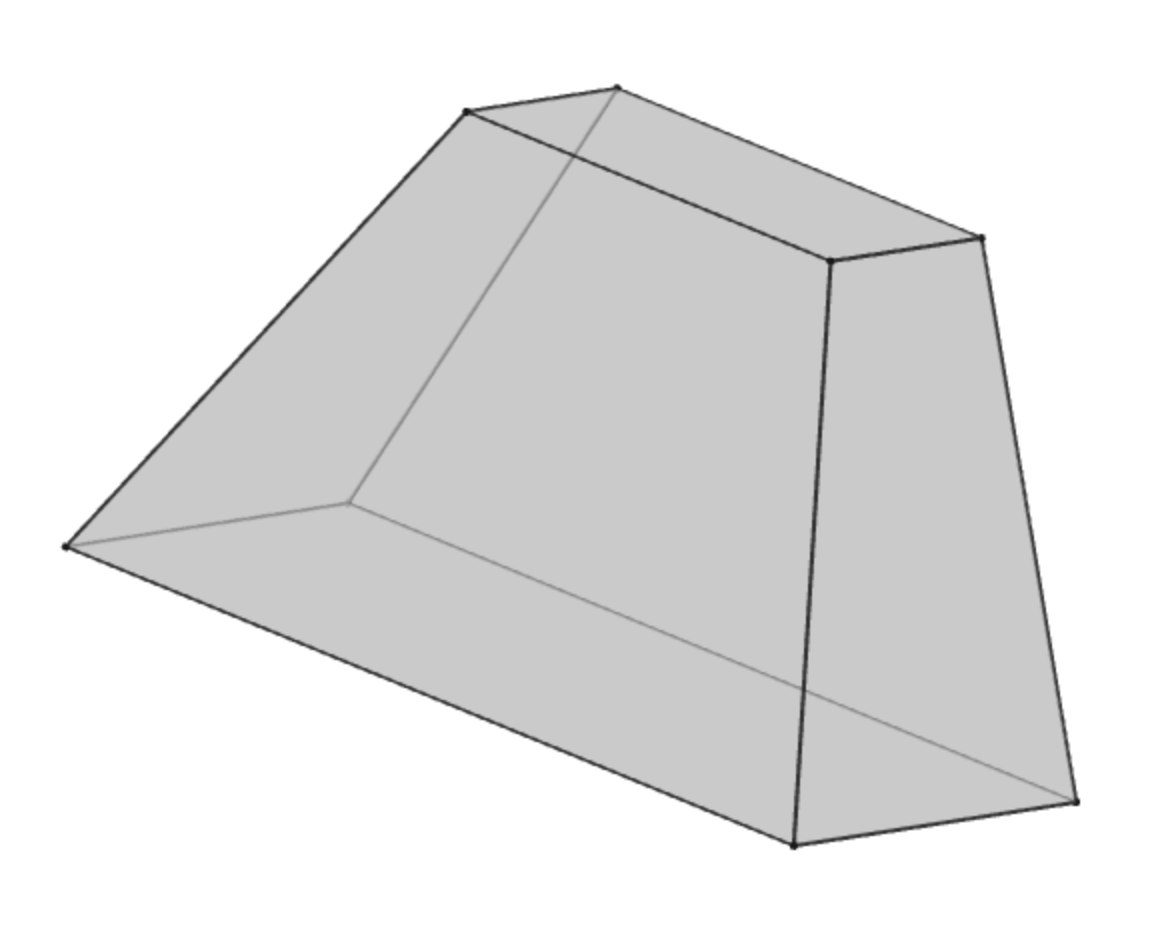}
        \caption{Prismoid or loft}
        \label{fig: prismoid}
    \end{subfigure}
    \caption{The polyhedrons we will find in this work}
    \label{fig: all polyhedrons}
\end{figure}

\subsection{A discussion on polyhedrons}

In this section we state a proposition on maps from certain polyhedral configurations onto segments. These segments will become the edges of the tree of our sphere decomposition. The result is proven in the appendix to do not be distracted from the flow of the argument.

\begin{definition}
    \label{def: configuration}
 We say two polyhedrons are in a \textbf{prism inside a prism configuration} if both polyhedrons are rectangular prisms, with parallel faces, and one is strictly inside the other. We denote this by \textbf{$PP$-configuration}.

In the above configuration there is exactly one polyhedron that wraps around the other. We call this the \textbf{surrounding or exterior} polyhedron of the configuration. We call the other polyhedron the \textbf{interior} ones. 
\end{definition}

The main result of this section follows.

\begin{proposition}
\label{prop: PP/LP existance of phi}
    Let $X\subset Y$ be in a $PP$-configuration. Define
    \begin{equation*}
        \mathcal{S}_{XY} := \overline{Y - X},
    \end{equation*}
    that is, the closed set wrapped between the two polyhedrons with the surfaces involved.
    There exists a continuous map $\Phi: \mathcal{S}_{XY} \longrightarrow [0, 1]$ such that
    \begin{enumerate}
        \item For each $0\le u \le 1$, $\Phi^{-1}(u)$ is a rectangular prism with faces parallel to those of $Y$. In particular, it is a polyhedral sphere $\mathbb{S}^2$.
        \item $\Phi^{-1}(0) = X$ and $\Phi^{-1}(1) = Y$.
    \end{enumerate}
    This results holds if $X$ is a point. The only change is the inequality in property $1$ should be $0 < u \le 1$.
\end{proposition}
\begin{proof}
    This is proven in the appendix.
\end{proof}
\begin{remark}
    We can change the interval $[0, 1]$ by any other closed, nondegenerate, interval $[a, b]$. We will often do this without further comment.
\end{remark}

The role of $\Phi$ is to parametrize the level sets as $X$ grows and becomes $Y$. In figure \ref{fig: L grows to P overall} we see a particular example. Notice how each of the level sets is a rectangular prism.

\begin{figure}[H]
    \centering
        \includegraphics[scale = 0.3]{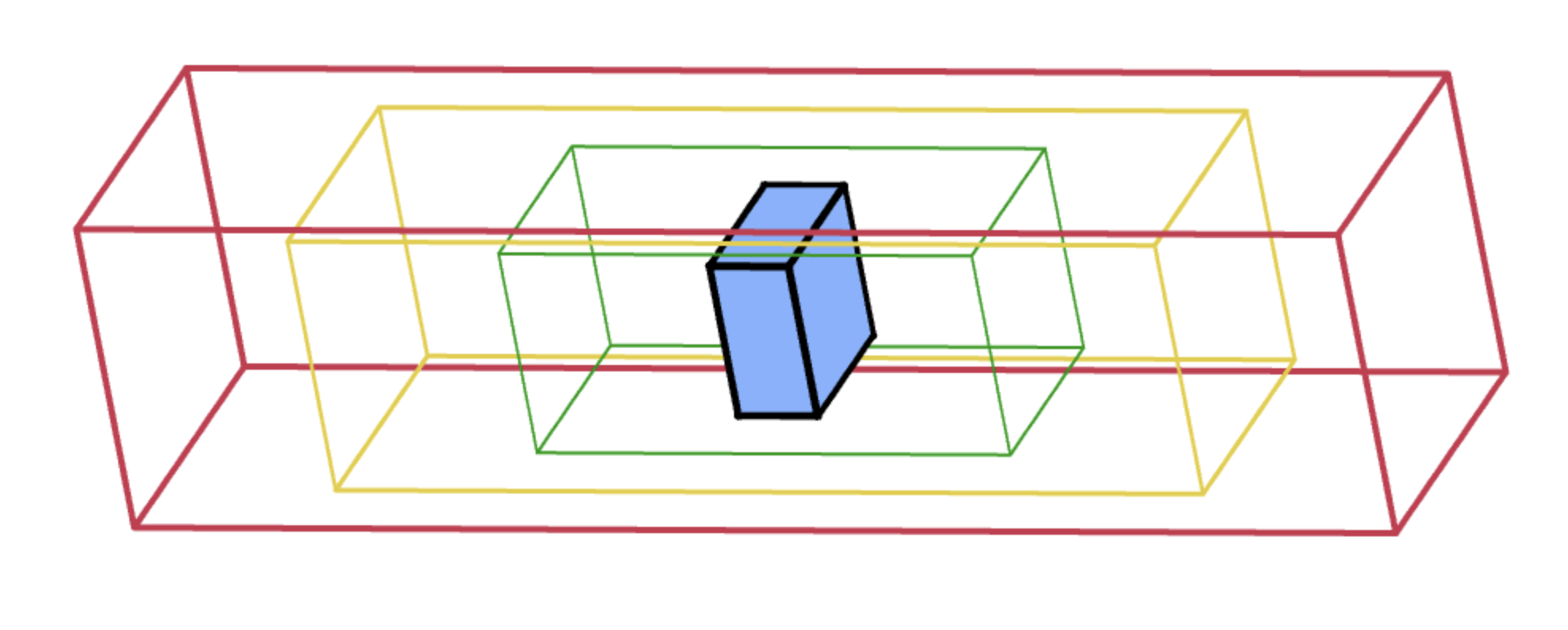}
        \label{fig: L grows to P separate}
    \caption{The level sets $\Phi^{-1}(0), \Phi^{-1}(1/3), \Phi^{-1}(2/3)$ and $\Phi^{-1}(1)$.}
    \label{fig: L grows to P overall}
\end{figure}

\begin{remark}
    The reason why proposition \ref{prop: PP/LP existance of phi} is not immediately obvious is because the rectangular prisms may not be homothetic or similar. Hence, one must take care of how to build the map.
\end{remark}

\section{Sphere decomposition and spherewidth}
\label{sec: Sphere decomposition and spherewidth}

In here we review the main tool for our purposes. We are following \cite{LunelSphereWidth}.

\subsection{Review of the main concepts: double bubbles}

 We begin with a definition.

\begin{definition}
    \label{def: double bubble}
A \textbf{double bubble} is a triple $(D_1, D_2, D_3)$ of closed disks in $\mathbb{S}^3$. They are disjoint except on their boundaries, where
\begin{equation*}
 D_1 \cap D_2 = D_1 \cap D_3 = D_2 \cap D_3 = D_1 \cap D_2 \cap D_3 = \partial D_1 = \partial D_2 = \partial D_3   
\end{equation*}
\end{definition}

In figure \ref{fig: generic double bubble} we can see a generic example of a double bubble. 
\begin{figure}[H]
    \centering
    \includegraphics[scale = 0.1]{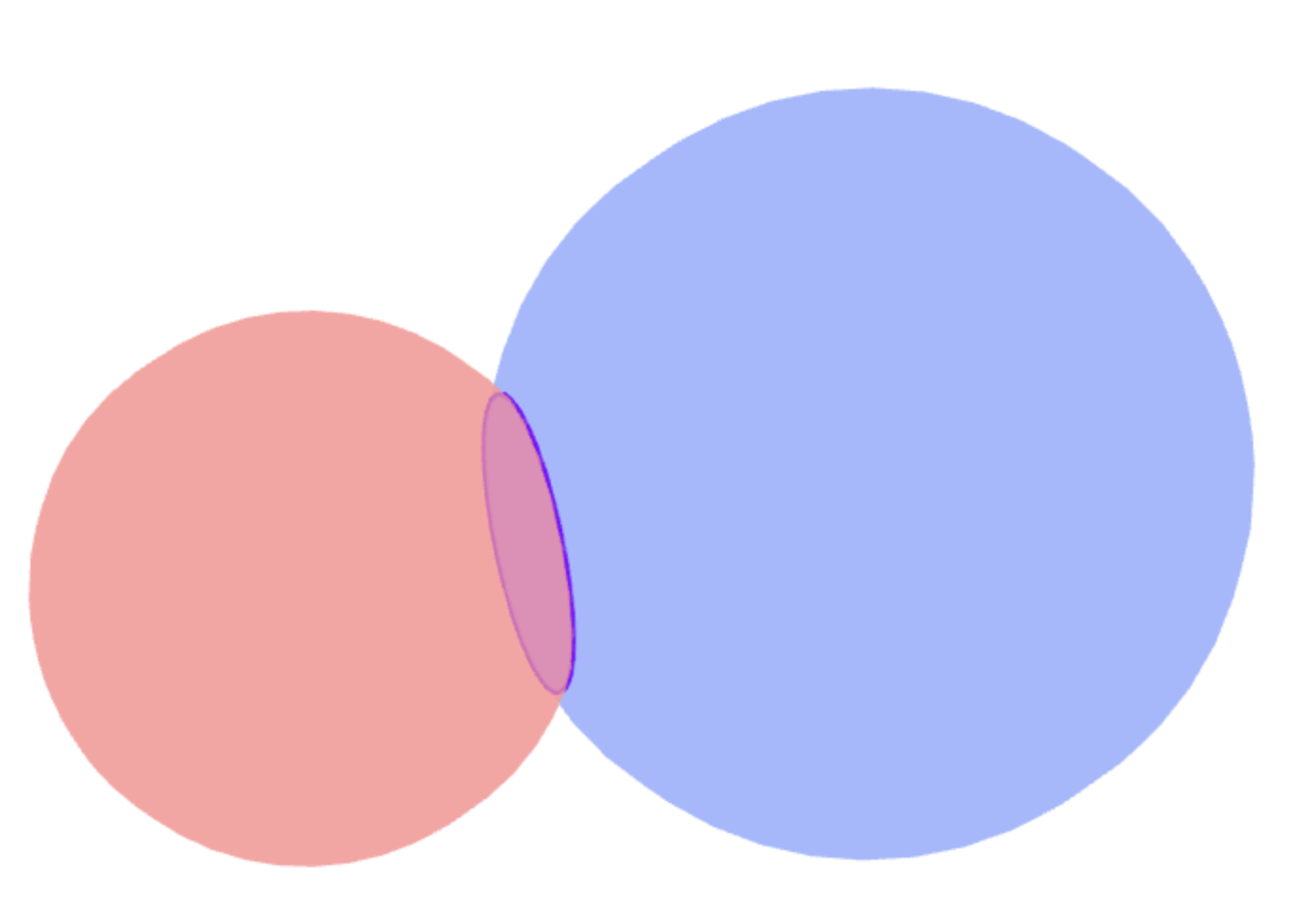}
    \caption{A generic double bubble.}
    \label{fig: generic double bubble}
\end{figure}
For our purposes, the way the double bubbles will be constructed is by two prisms that \say{merge} into a bigger prism. We can see the two situations in figure \ref{fig:doublebubbles}.

\begin{figure}[H]
    \centering
        \includegraphics[scale = 0.1]{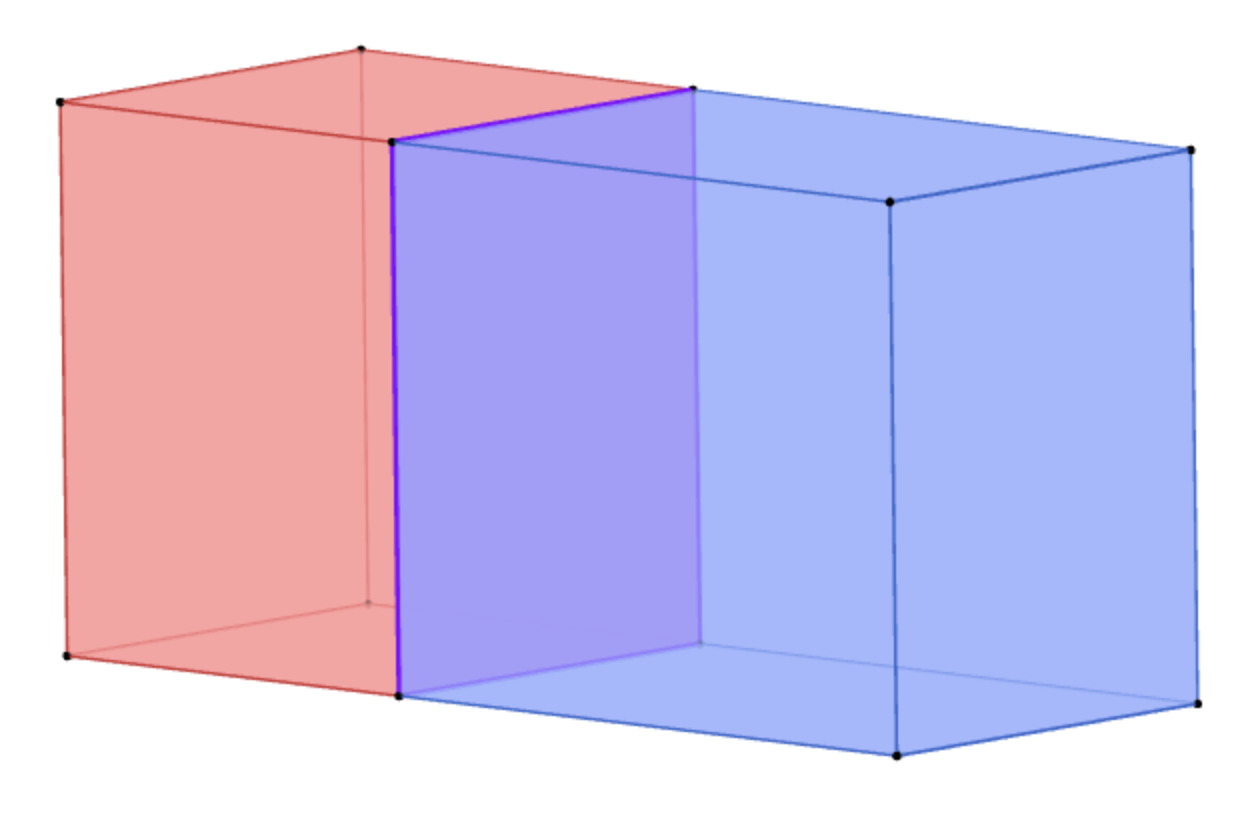}
    \caption{The double bubbles we will find in our work.}
    \label{fig:doublebubbles}
\end{figure}

\subsection{Review of the main concepts: thickening and transversality}

The following definitions are taken from \cite[section 2]{LunelSphereWidth}.

\begin{definition}
\label{def: thickened embedding}
    A \textbf{thickened embedding} of a graph $G$ is an embedding of $G$ in $\mathbb{S}^3$ where each vertex is thickened to a small ball. Two balls are connected by a polygonal edge if an only if they are adjacent in the graph $G$, and pairs of edges are disjoint. 
\end{definition}
\begin{remark}
    We emphasize that the edges themselves are not thickened.
\end{remark}

The main object for our work, which we will shortly see, requires that we work with the above type of graphs. For us the graph will be one of the $1$-skeletons of the finite iterations or the path the knot follows on it. In particular, the need for thickening will require from us to be careful on how the vertices are thickened to spheres.

\begin{definition}
    \label{def: transversal definition}
\begin{enumerate}
    \item Two surfaces embedded in $\mathbb{S}^3$ are \textbf{transverse} if they intersect in a finite number of connected components, where the intersection is locally homeomorphic to the intersection of two orthogonal planes.

    \item A knot and a surface are \textbf{transverse} if they intersect in a finite number of connected components, where the intersection is locally homeomorphic to the intersection of a plane and an orthogonal line.

    \item A surface and a ball are \textbf{transverse} if the surface and the boundary of the ball are transverse.

    \item A surface and a graph are \textbf{transverse} if the surface is transverse to all the thickened vertices and edges it intersects.

    \item A double bubble and a graph or surface are \textbf{transverse} if each of the three spheres it defines is and if the vertices of the graph do not intersect the spheres on their shared boundaries.
\end{enumerate}
\end{definition}

The concept of transversality, as defined above, is an adaptation of the usual notion of transversal intersections, to cover also the case of thickened embeddings. The opposite concept to it is that of tangency, which we also define to be precise in the context we need it.

\begin{definition}
    \label{def: tangency}
\begin{enumerate}
    \item Intersections between pairs of objects as the ones discussed in definition \ref{def: transversal definition} are said to be \textbf{tangent} when they are not transversal.

    \item A sphere $S$ and a graph $G$ are said to be \textbf{finitely tangent} when they do not intersect transversely but the number of intersections is finite.
\end{enumerate}  
\end{definition}

Through our work, where we will deal with the above concepts repeatedly, we will have to be careful of checking that our intersections are either transversal or finitely tangent.

\subsection{Review of the main concepts: sphere decompositions}

The following is \cite[Definition 2.1]{LunelSphereWidth}. We recall that for a graph $G$ we denote its set of leaves and vertices by $L(G)$ and $V(G)$, respectively. 

\begin{definition}
    \label{def: sphere decomposition}
Let $G$ be a graph thickly embedded in $\mathbb{S}^3$. A \textbf{sphere decomposition} of $G$ is a continuous map $f:\mathbb{S}^3\longrightarrow \cT$ where $\cT$ is a trivalent tree with at least one edge. It satisfies
\begin{description}
    \item[(SD 1)] For all $x\in L(\cT)$, $f^{-1}(x)$ is a point disjoint from $G$.

    \item[(SD 2)] For all $x\in V(\cT)-L(\cT)$ , $f^{-1}(x)$ is a double bubble transverse to $G$.

    \item[(SD 3)] For all $x\in \cT$ interior to an edge, $f^{-1}(x)$ is a sphere transverse or finitely tangent to $G$.
\end{description}
\end{definition}
\begin{remark}
    We emphasize that in the above, the graph $G$ is thickened, so the vertices are balls and the concept of transversality is related to intersections with its surface.
\end{remark}

Associated to a given sphere decomposition, we have its weight, which is a way to measure how much the knot spreads through the different spheres in the decomposition. First, we have the following general concept.
\begin{definition}
    \label{def: weigh of S with respect to G}
Let $G$ be a thickly embedded graph and $S\subset\mathbb{S}^3$ a sphere. The \textbf{weight of $S$ with respect to $G$} is the number of connected components of $S\cap G$. We denote if by $w(S, G)$ or $w(S)$ when there is no possibility of confusion.
\end{definition}

As promised, we now have the following 
\begin{definition}
    \label{def: width}
Let $G$ be a graph. Let $f:\mathbb{S}^3\longrightarrow \cT$ be a sphere decomposition of a thick embedding of $G$. We define \textbf{the width of $f$} as the supremum of the widths of the spheres $f^{-1}(x)$ for all $x$ interior to the edges. That is,
\begin{equation*}
    w(f) := \displaystyle\sup_{x\in G - V(G)} w(f^{-1}(x)).
\end{equation*}
Furthermore, the \textbf{spherewidth} of $G$ is the infimum of the widths of its sphere decompositions. That is,
\begin{equation*}
    sw(G) := \displaystyle\inf_{f:\mathbb{S}^3\longrightarrow \cT} w(f),
\end{equation*}
where $f$ runs over all sphere decompositions of $G$.
\end{definition}
\begin{remark}
    For any explicit sphere decomposition $f_0:\mathbb{S}^3\longrightarrow \cT$ of $G$, we have
    \begin{equation*}
        sw(G) \le w(f_0).
    \end{equation*}
\end{remark}
\begin{remark}
    The spherewidth depends on $G$ but not in any particular thick embedding. Rather, the infimum takes all of those into account. This is similar to how knot invariants defined over knot diagrams depend on the knot, but not on the diagram itself. 
\end{remark}

\section{Sphere Decomposition induced by the Sierpinski tetrahedron}
\label{sec: Sphere Decomposition induced by the Sierpinski tetrahedron}
\subsection{The change of $1$-skeleton}

Throughout this section we will suppose we have given a knot $K$ embedded on the one skeleton of the $n$-th iteration of the Sierpinski tetrahedron. 

For purposes of our result, it will be useful to modify slightly the construction of the $1$-skeletons. The finite iterations are obtained by subsequently removing part of regular tetrahedrons. We can modify the iteration process by not substituting the pyramid with four copies of itself with ratio $1/2$ but with any other smaller ratio $\varepsilon > 0$. 

If, while repeating this iteration, we leave the edges of previous iterations we get a one skeleton with extra edges. Indeed, each vertex of the finite iteration, where two of the smallest pyramids meet, corresponds to an edge whose endpoints are the vertices. We can compare how the second iterations of both processes look in figure \ref{fig: comparing two sierpinskis}. A factor of $1/4$ was used for the modified pyramid construction.

\begin{figure}[H]
    \centering
    \begin{subfigure}{0.48\textwidth}
        \centering
        \includegraphics[scale = 0.2]{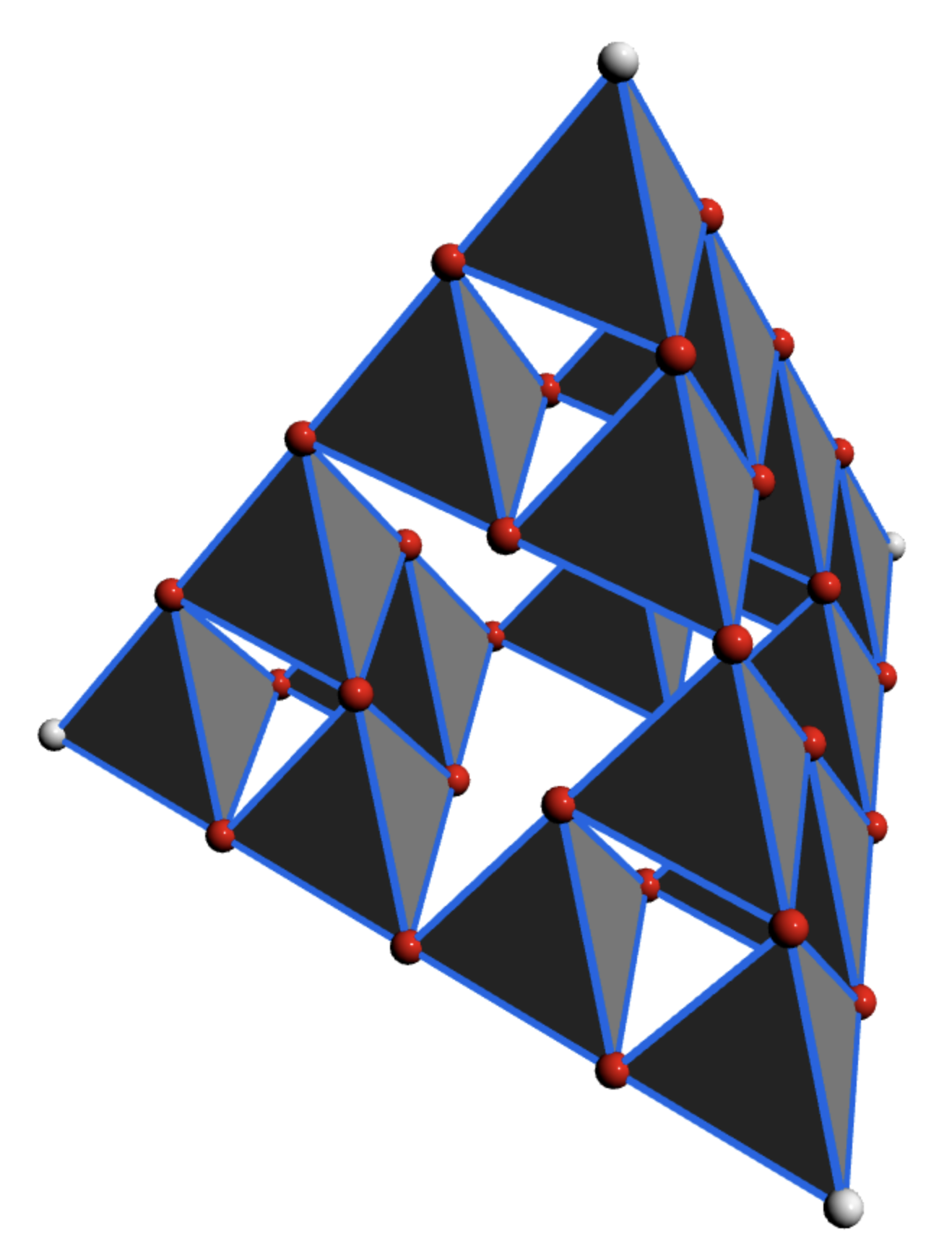}
        \caption{Second iteration for the Sierpinski tetrahedron.}
        \label{fig: regular Sierpinski}
    \end{subfigure}
    \hfill 
    \begin{subfigure}{0.48\textwidth}
        \centering
        \includegraphics[scale = 0.2]{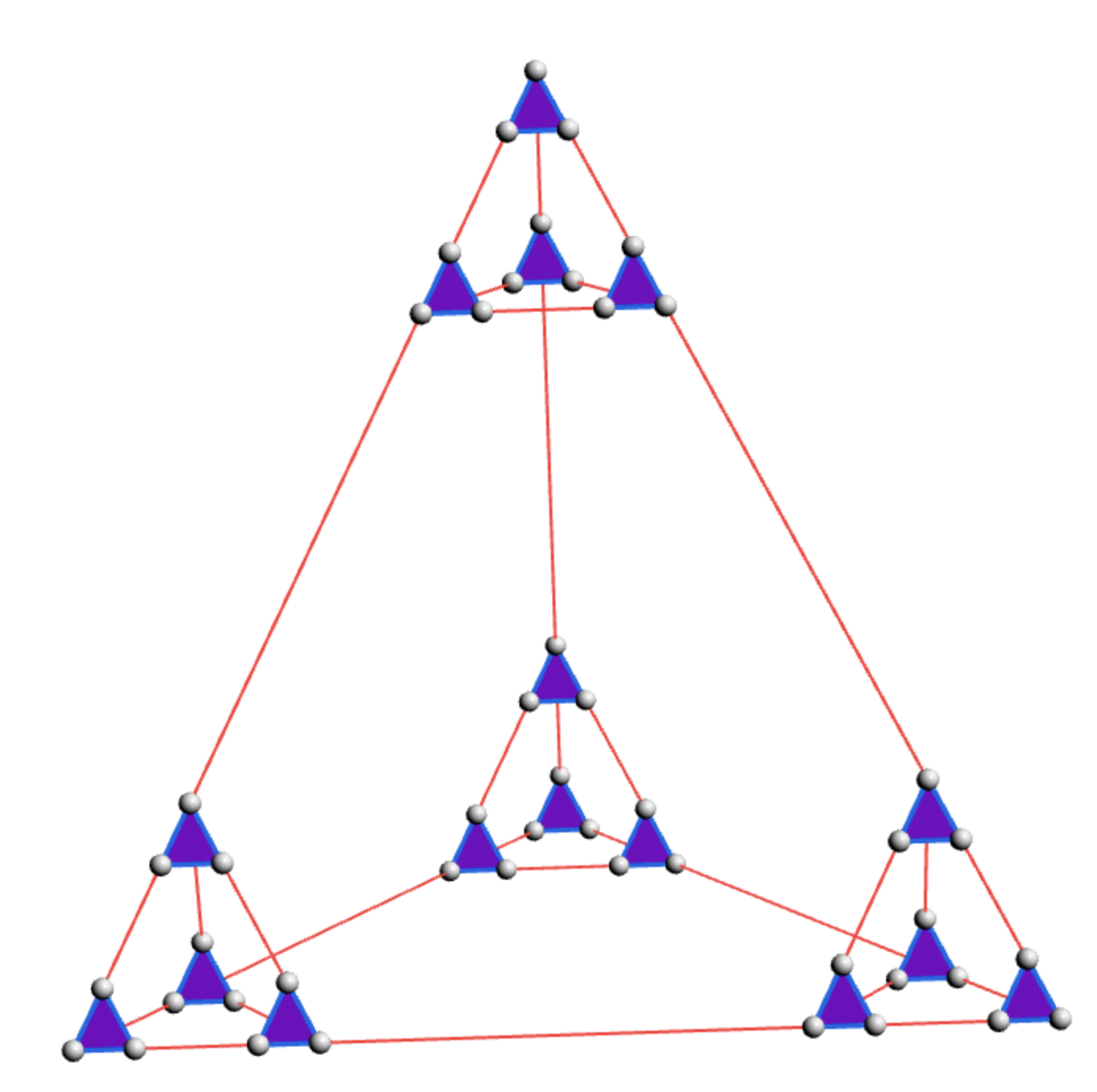}
        \caption{Second iteration for the modified construction.}
        \label{fig: Right Angle Sierpinski}
    \end{subfigure}
    \caption{Comparison of the one skeletons of both constructions.}
    \label{fig: comparing two sierpinskis}
\end{figure}
The important thing to notice is that, even though the graphs are different, the knots embedded in them are the same because those extra edges can be contracted back to a point. In figure \ref{fig: Right Angle Sierpinski} we see the red edges correspond to the red vertices of \ref{fig: regular Sierpinski}. Only the original vertices do not change to an edge because they do not touch another pyramid there.

Finally, we can enclose each of the tetrahedrons in a cube whose faces are transversal to the edges of the $1$-skeletons. Indeed, the edges of the $1$-skeleton comprise six different directions. Consequently, this requirement is the complement of a finite number of linear conditions and thus achievable. By enlarging the cubes slightly we can suppose the edges intersect the faces of the cube in the interior of the faces and not on edges or vertices. 

\begin{figure}[H]
    \centering
    \begin{subfigure}{0.48\textwidth}
        \centering
        \includegraphics[scale = 0.2]{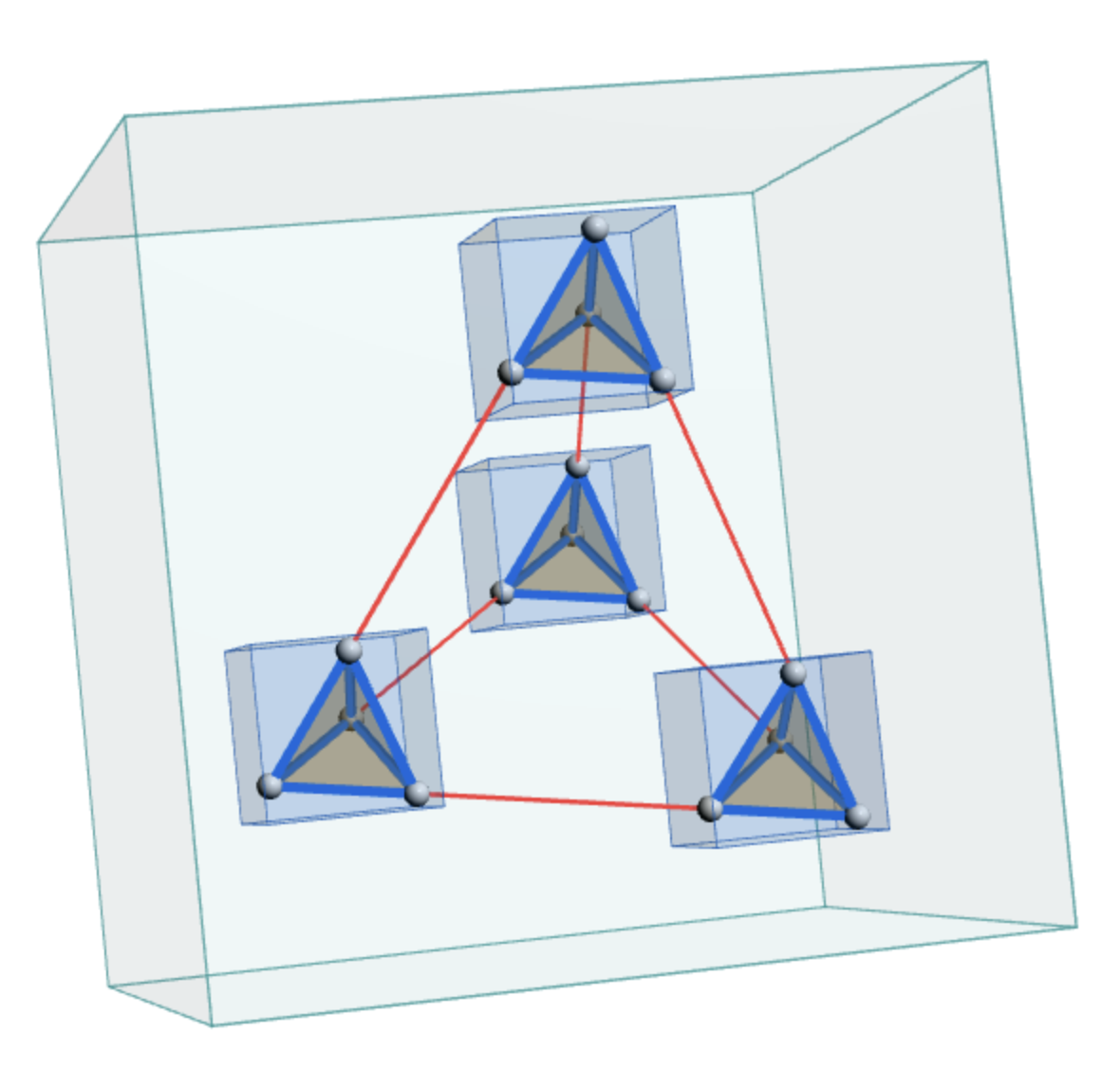}
        \caption{The configuration after a single iteration.}
        \label{fig: Fixed Config First}
    \end{subfigure}
    \hfill 
    \begin{subfigure}{0.48\textwidth}
        \centering
        \includegraphics[scale = 0.2]{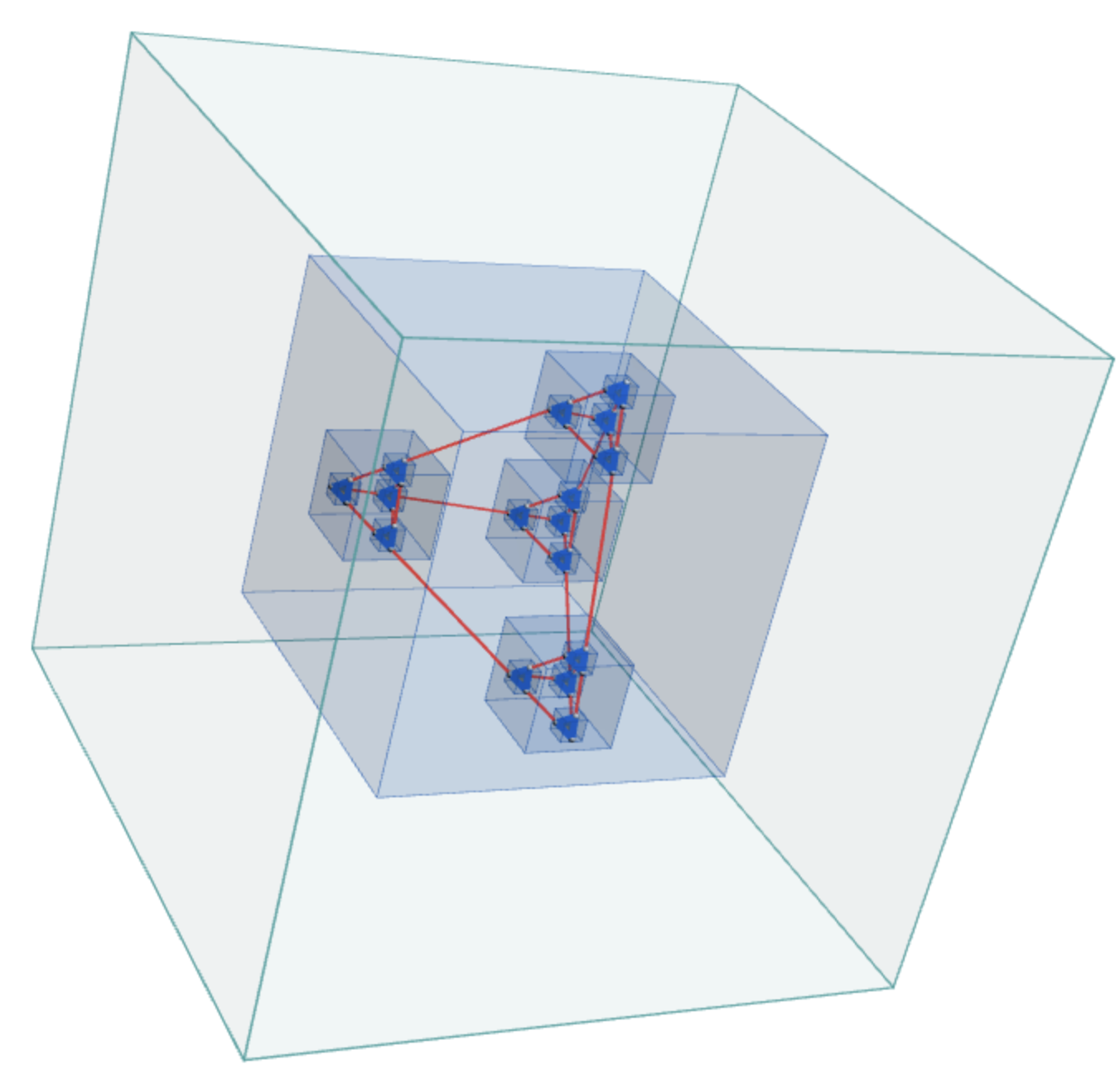}
        \caption{The configuration after the second iteration}
        \label{fig: Fixed Config Second}
    \end{subfigure}
    \caption{Configurations with pyramids inside cubes and replaces slanted edges.}
    \label{fig: Overall Fixed Configuration}
\end{figure}

By iterating this $n$ times we can assume we are working with a knot embedded on the $1$- skeleton of a graph constructed as above.

We now introduce a definition that will make our description easier to follow.

\begin{definition}
    \label{def: smallest tetrahedron}
    We say a tetrahedron $\Delta$, of any iteration, is a \textbf{smallest tetrahedron} if 
    \begin{equation*}
        \emptyset \neq K \cap \Delta \subseteq (1-\text{skeleton of }\Delta).
    \end{equation*} 
    In other words, of $\Delta$, $K$ only uses its edges but never goes into the faces or interior. 
    
    We say a cube is a \textbf{smallest cube} if the biggest tetrahedron in its interior is a smallest tetrahedron.
\end{definition}

Our last enhancement of our construction is to surround each vertex of the smallest tetrahedrons with a cube, parallel to all the cubes in our construction, disjoint among themselves and within the interior of the corresponding smallest cube. 

\begin{definition}
    \label{def: vertex cubes}
    The cubes just constructed will be called \textbf{vertex cubes}.
\end{definition}

\subsection{The thickening of the vertices}

Definition \ref{def: sphere decomposition} of sphere decomposition involves a thickened embedding of $K$. So far, we have only modified the $1$-skeleton so that we are able to surround individual tetrahedrons in their own cubes as shown in figure \ref{fig: Overall Fixed Configuration}. In this subsection we briefly discuss the vertices. As we shall see, we can thicken them with a small polyhedron as long as we have control of a particular subtlety that arises.

In the definition of sphere weight, we intersect the sphere $S$ with the thickened graph and count the connected components of the intersection. In the sphere decomposition we will build, all our spheres will be rectangular prisms whose faces are parallel to three orthogonal fixed planes. Consequently, we must understand how these families of planes intersect the \textit{thickened} vertex 

As we sweep with planes that pass through the vertex, as if they were scanning the zone the vertex lies in, what we desire is to obtain a count of how the edges grow and merge into a corner. However, we must also count the intersection of this sweeping plane with the sphere. 

\begin{figure}[H]
    \centering
    \includegraphics[scale = 0.2]{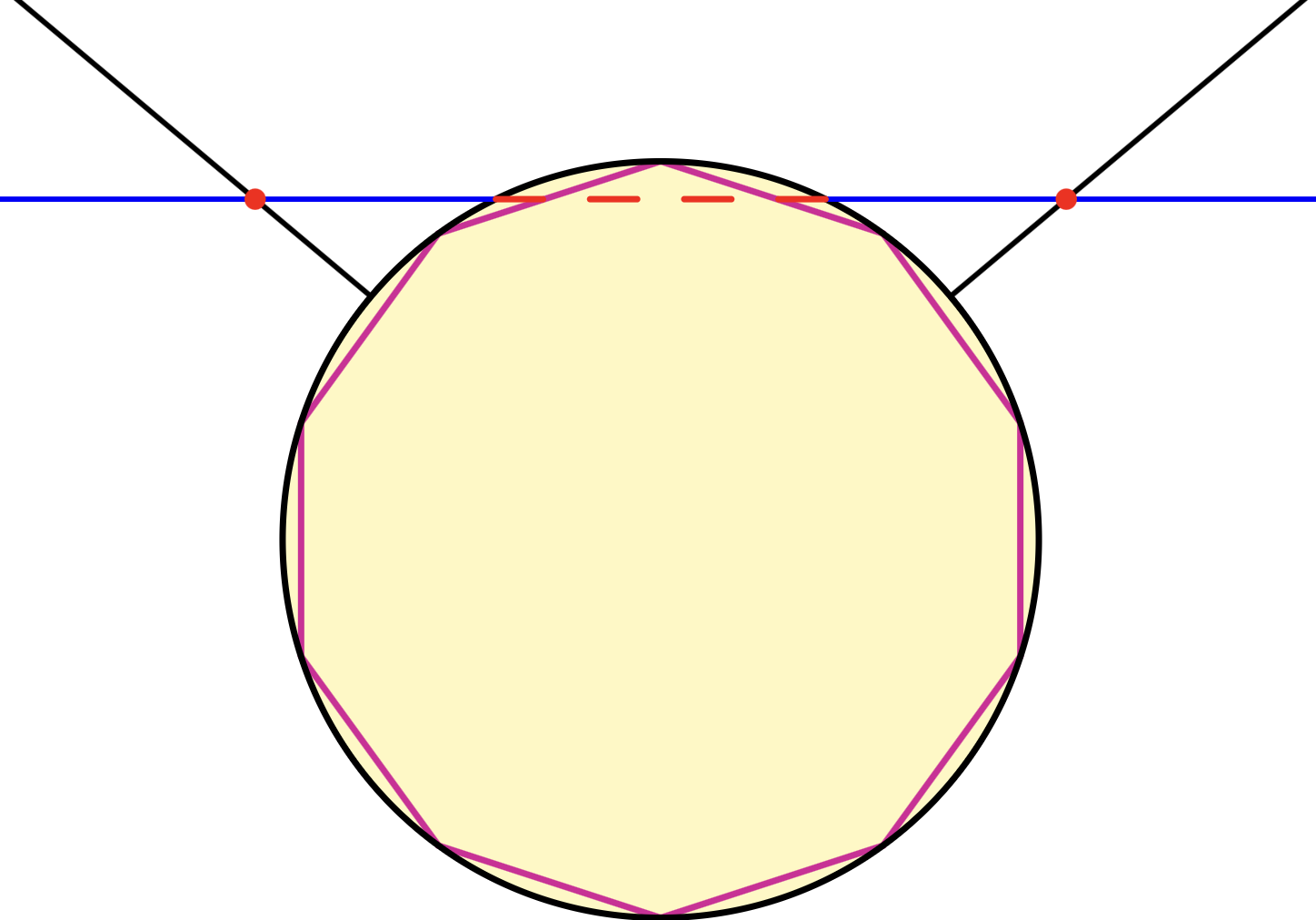}
    \caption{Artificial intersection with a sweeping plane}
    \label{fig: artificialintersection}
\end{figure}

As we can see in figure \ref{fig: artificialintersection}, as we approach the vertex we intersect the sphere in a component that is artificial. It will eventually merge with the other points on the edges but, nevertheless, made the weight of the sphere (i.e. the sweeping plane) to grow. We cannot disappear this problem by choosing a small enough sphere or similar. These intersections are features of the round shape of the sphere (or the polyhedron). Consequently, to prove an upper bound on the spherewidth, we must guarantee our thickening does not make the weight grow artificially beyond our desired bound. We will prove in due time our sphere decomposition satisfies this.

\begin{remark}
  There is another possibility to handle this. We can design an appropriate piece so that, if we are careful with our construction of the tree decomposition, these artificial intersections do \textit{not} happen. This leads to the theory of separating polygons, Dirichlet - Voronoi cells, etc. 
  
  Thus, there is a dichotomy: we can let them happen but with bounded contribution or we can eliminate them. Once we make this choice, we design the tree decomposition.   

\end{remark}

With this in mind, we put around each vertex of each smallest pyramid a thickening polyhedron contained in the vertex cube. At the last step of our calculation, we will reduce the size of each of these polyhedrons. We will see this doesn't create conflict with any construction as we can put the correct ones from the very start once we know what extra property we require from them.

\subsection{The subdivision}

Once the above has been carried out we have the following data: 
\begin{enumerate}
    \item A finite number of cubes arranged in the way described above. Each cube, except for the vertex cubes, has inside exactly 4 smaller cubes arranged as in figure \ref{fig: Fixed Config First}.
    \item A cycle through some of the edges described above that knots as $K$. 
\end{enumerate}

All the above information induces a subdivision into closed sets of $\mathbb{S}^3$ as follows:
\begin{description}
    \item[The exterior] The closure of the exterior of the largest cube. This includes the point at infinity and is homeomorphic to a ball $\mathbb{B}^3$.

    \item[The intermediate regions] Each cube $C$, except for the vertex cubes, has up four disjoint cubes inside of it. 
    
    However, we shall only consider the cubes if the knot $K$ intersects it. Otherwise we do not collocate a cube there as we can remove that part of the skeleton without altering the knot embedding. If $C_1,..., C_k$ are the remaining cubes, the closed set is
    \begin{equation*}
        \overline{C-(C_1\cup \cdots \cup C_k)}.
    \end{equation*}

     \item[The vertex cubes] A finite number of vertex cubes that inside hold a thickened vertex that belongs to the path traced by $K$. The other vertex cubes, that were not visited, and the edges connecting to them are erased, without modifying the path $K$ traces. 
\end{description}

\begin{example}
    \label{ex: intermediate region}

A generic intermediate region can be seen in figure \ref{fig: intermediate region}.
    \begin{figure}[H]
        \centering
        \includegraphics[width=0.4\linewidth]{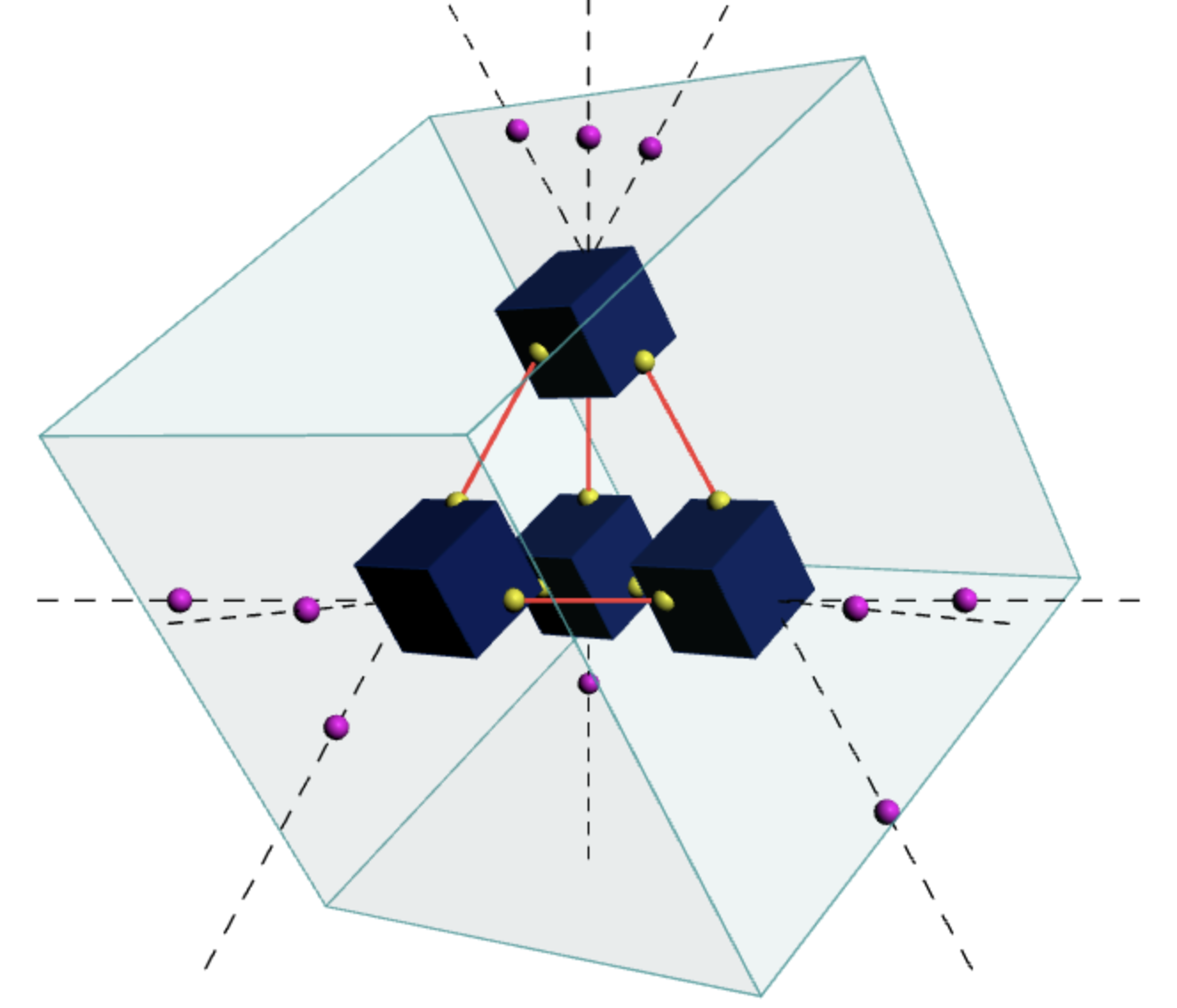}
        \caption{An intermediate region}
        \label{fig: intermediate region}
    \end{figure}
Notice that, except for cubes that surround one of the corners of the original pyramids, every cube has a way to go out of $C$ because it must connect to some other pyramid in another intermediate region. This possible way outs are marked by the purple vertices and dashed lines.

From each interior cube there is only \textbf{one} connecting line outside. The three dashed lines presented for each interior cube show the possibilities for this connecting line. Which one it is depends on the location of the cube within the pyramid. Furthermore, the cubes that contain the corners of the initial tetrahedron do not have any connecting lines to the exterior of the surrounding cube.

On the other hand, in figure \ref{fig: intermediate region}, we see that the yellow points on a cube are potential intersections of the knot with the surface of these interior cubes. It is possible a knot intersects in all of them, some of them, or that it uses none in case that cube is not visited. In the latter case, we have agreed to not put any cube and remove all the part of the $1$-skeleton that would be there.

We emphasize that figure \ref{fig: intermediate region} continues to hold for the smallest regions, where these cubes are the vertex cubes (i.e. the interior cubes we have thickened vertices as opposed to another tetrahedron)
\end{example}

The strategy we will follow is the next one: for each one of these closed sets $\cC$ we will construct a continuous map $f_{\cC}: \cC \longrightarrow \cT_{\cC}$, where $\cT_{\cC}$ is a trivalent tree. We will then glue the trees $\cT_{\cC}$ to form a trivalent tree $\cT$. Finally, the gluing lemma will imply the corresponding map $f: \mathbb{S}^3\longrightarrow \cT$ is continuous. Once we have done this, we shall verify $f$ is a sphere decomposition.

The following is the fact which makes the whole computation work.

\begin{proposition}
    \label{prop: bound on cubes}

Let $C$ be any of the cubes of the above described subdivision. 
\begin{enumerate}
    \item The intersection $\partial C \cap K$ has at most four connected components of $K$ (all them points). 
    \item The intersection $C\cap K$ has at most two connected components.
\end{enumerate}
\end{proposition}
\begin{proof}
    By construction, to cross $\partial C$, the knot must go through one of the dashed lines towards the outside. However, per interior cube there is at most one way out, as each pyramid connects to at most one other pyramid at each of its vertices. As there are four vertices, this number is at most four. This proves $1$.

    If $C$ is the outer surrounding cube this is obvious as the knot is connected. For other cubes, if the knot visits the given cube then, it must go out of it. As we have establish to go in or out one must visit one of the four points on $\partial C$. Hence, there can be at most two pairings, of these four vertices, as entry-exit pairs. This proves $2$. 
\end{proof}

\subsection{The tree for an intermediate region}

In this subsection we construct a trivalent tree associated to each intermediate region $\cC$. 

\begin{lemma}
    \label{lemma: existence of intermediate map}
Let $\cC$ be an intermediate region. There exists a trivalent tree $\cT_{\cC}$ and a continuous map $f_{\cC}: \cC \longrightarrow \cT_{\cC}$ that satisfies the following properties:
\begin{enumerate}
    \item For each vertex $x$ that is not a leaf, the preimage $f^{-1}(x)$ is a double bubble transversal to $K\cap \cC$.

    \item For each point $x$ that is not a vertex, the preimage $f^{-1}(x)$ is a sphere transversal to $K$.

    \item The tree $\cT_{\cC}$ has a single root corresponding to its surrounding cube and one leaf for each one of its interior cubes. 

    \item Let $x_1,x_2 \in \cT$ two points that are \textit{not} trivalent vertices. If they are in the same edge, then 
    \begin{equation*}
        w(f^{-1}(x_1), K) = w(f^{-1}(x_2), K).
    \end{equation*}
\end{enumerate}
\end{lemma}
\begin{proof} 
    For concreteness, let us call the interior cubes $A, B, C$ and $D$ and the surrounding cube $E$. The technique is extremely simple: we will merge the cubes progressively. The process to do it will be to collocate the two chosen prisms to be merged into a third one and then invoke proposition \ref{prop: PP/LP existance of phi}.

   We begin with an observation. We can assume that our initial configuration, that then gets repeated at each iteration, has all of its interior cubes small enough so that no plane parallel to the faces of the cubes intersects two of the interior cubes simultaneously. In particular, in between any two cubes there is a separating plane in every direction parallel to the faces of the cube
   
   This is because we can pick the scalar factor of the Sierpinski tetrahedron to be as small as needed. Hence, the cube around it becomes as small as required to give the above separations. We only need this to work for one iteration. For all subsequent iterations we use the same configuration, under the appropriate scaling. 
   
   This allow us to guarantee that all prisms we will construct do not intersect more cubes than those they surround. This also implies that all edges they intersect have one of their endpoints in one of the prisms they surround. Indeed, if both endpoints were on the other side of the separating plane, by convexity of half-spaces, all the edge is.

    We will \say{merge} $A$ and $C$. To do this, we put them inside a common prism $R$, whose faces are parallel to the faces of the cubes.Then we choose the membrane between them in such a way that it only intersects the edge between $A$ and $C$ (this intersection is unavoidable). Then we repeat this construction with $R$ and $D$, and call the surrounding prism $Q$. We conclude by repeating this with $Q$ and $B$, and call the surrounding prism $U$. We can see this in figure \ref{fig: Merging}.

    \begin{figure}[H]
    \centering
    \begin{subfigure}{0.32\textwidth}
        \centering
        \includegraphics[scale = 0.2]{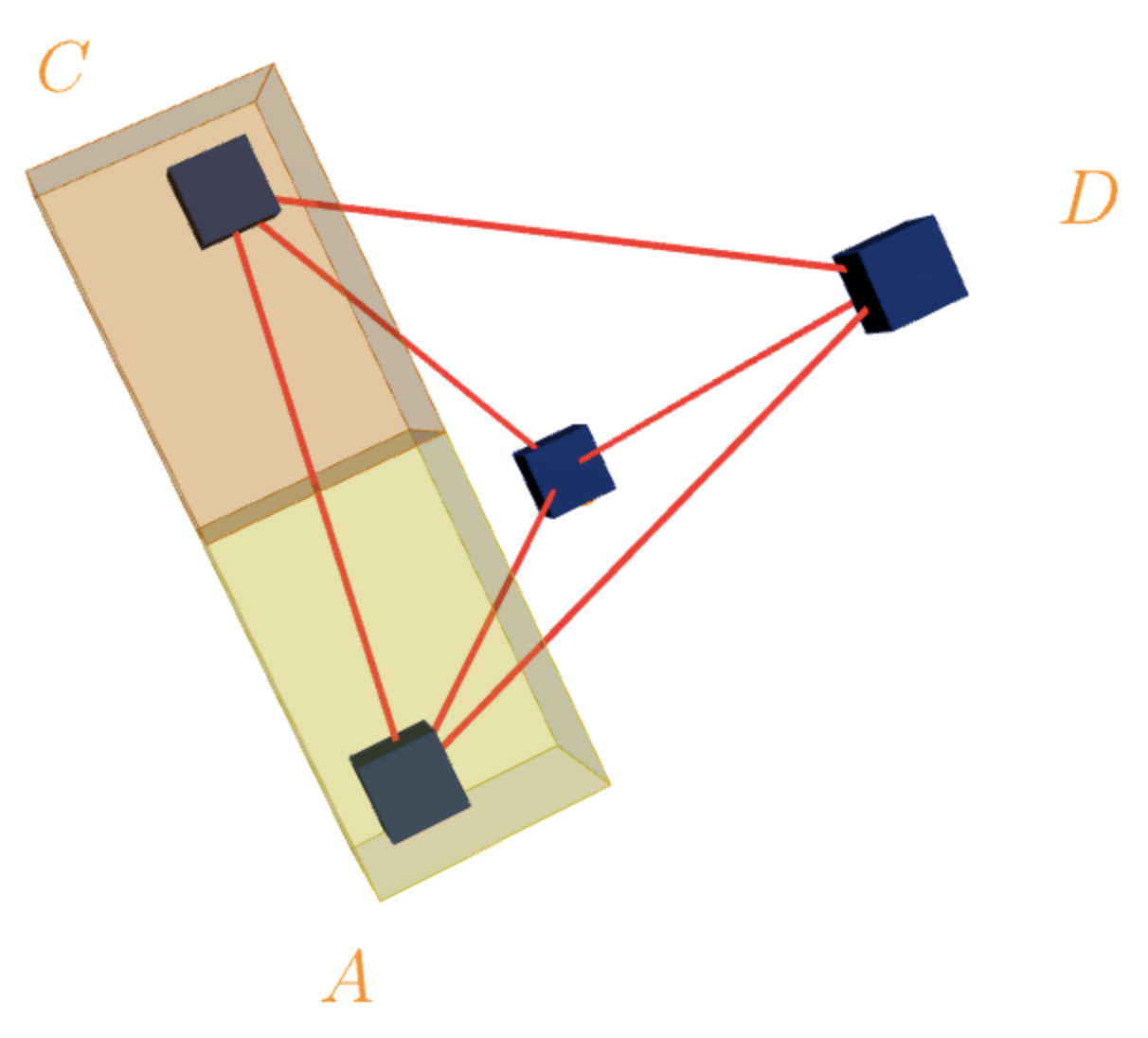}
        \caption{$A$ and $C$ merge.}
        \label{fig: ACmerge}
    \end{subfigure}
    \hfill
    \begin{subfigure}{0.32\textwidth}
        \centering
        \includegraphics[scale = 0.2]{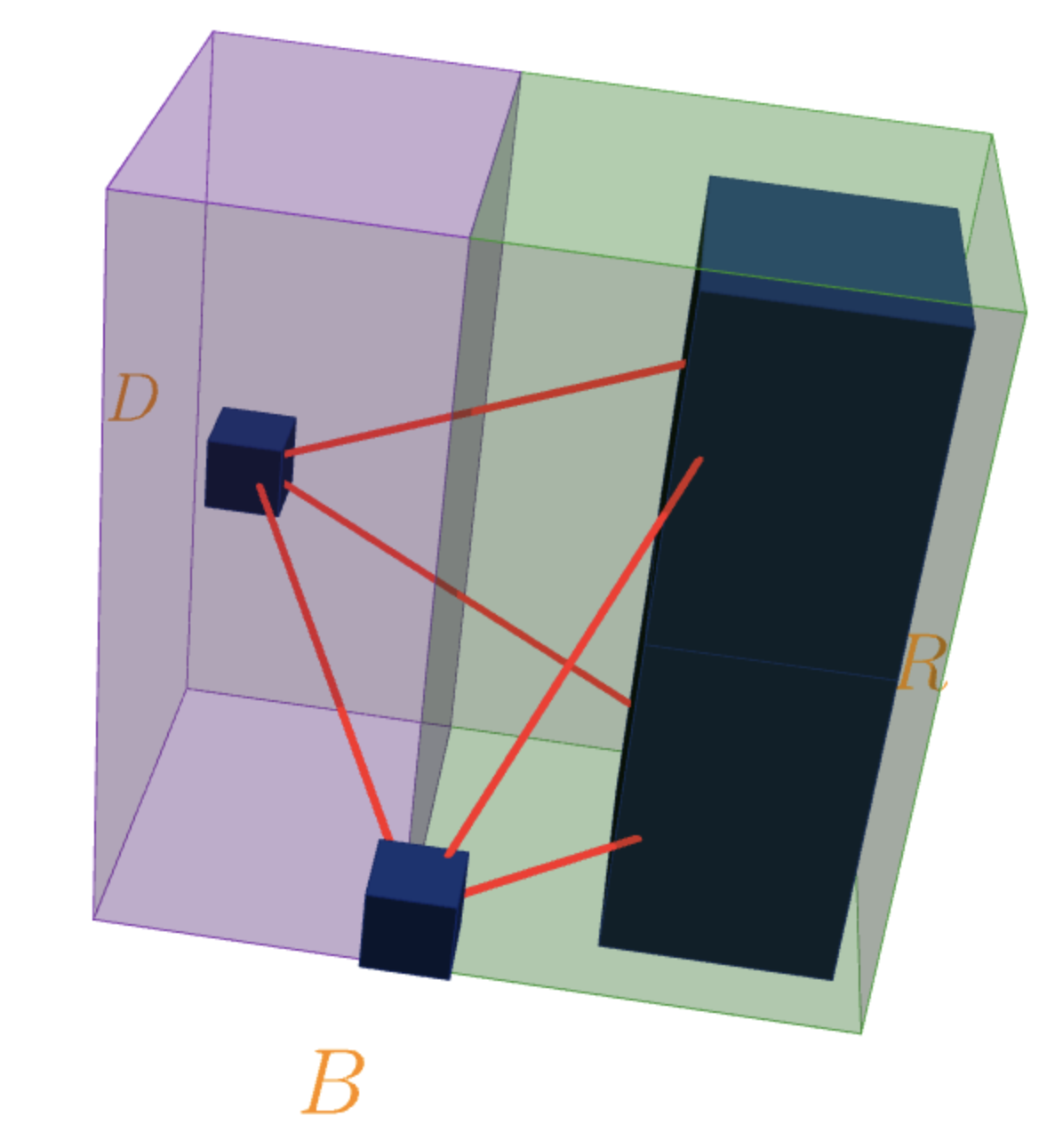}
        \caption{$R$ and $D$ merge.}
        \label{fig: ACDmerge}
    \end{subfigure}
    \hfill
    \begin{subfigure}{0.32\textwidth}
        \centering
        \includegraphics[scale = 0.2]{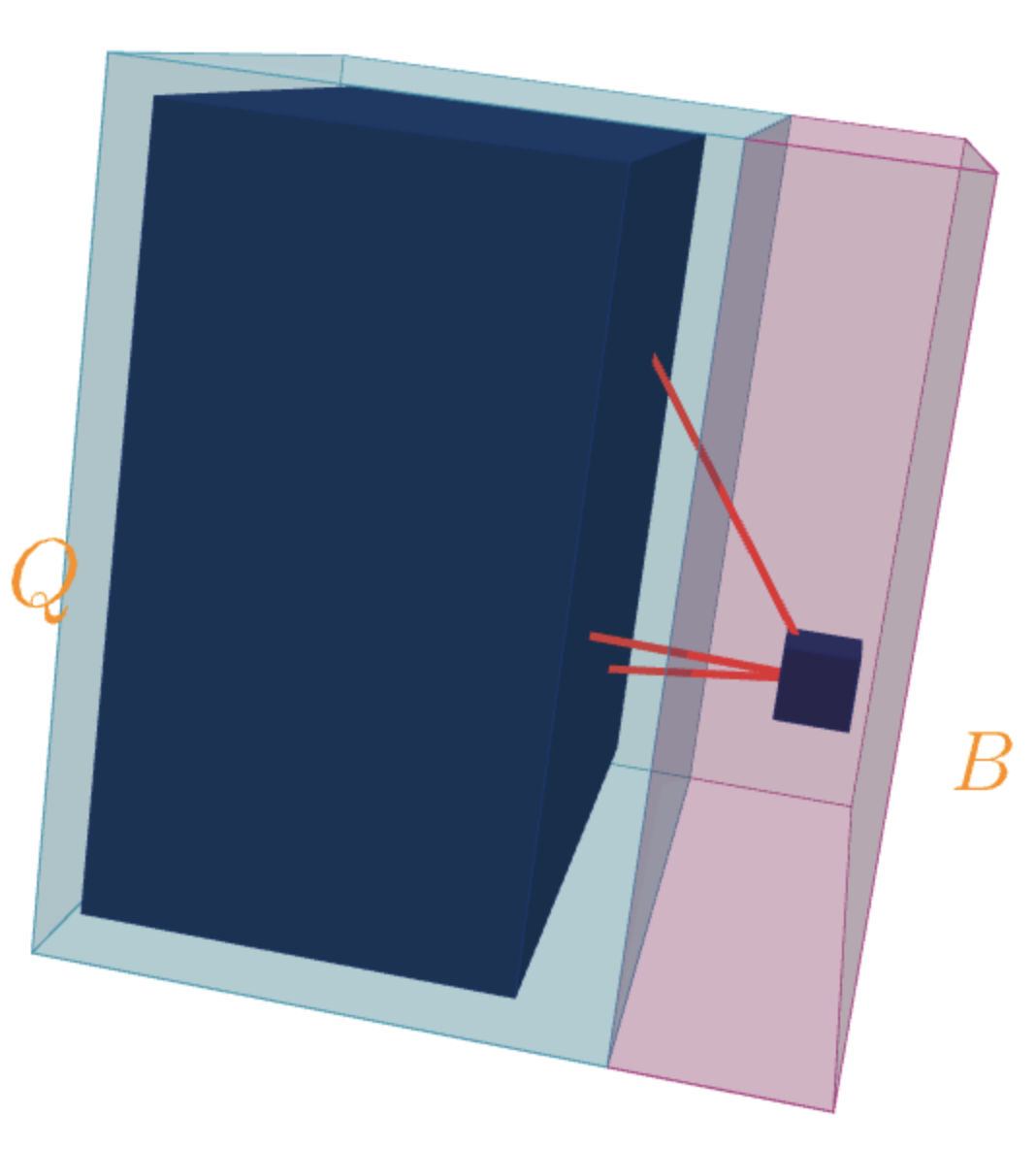}
        \caption{$Q$ and $B$ merge.}
        \label{fig: ABDCmerge}
    \end{subfigure}
    \caption{Process of progressive merging of the cubes.}
    \label{fig: Merging}
\end{figure}

    Once more, we emphasize this construction is possible because we can guarantee the edges of the merging prisms are far away from the faces of the other cubes that are not being merged. We can see this entire decomposition in figure \ref{fig: decomposition}.

    \begin{figure}[H]
        \centering
        \includegraphics[width=0.3\linewidth]{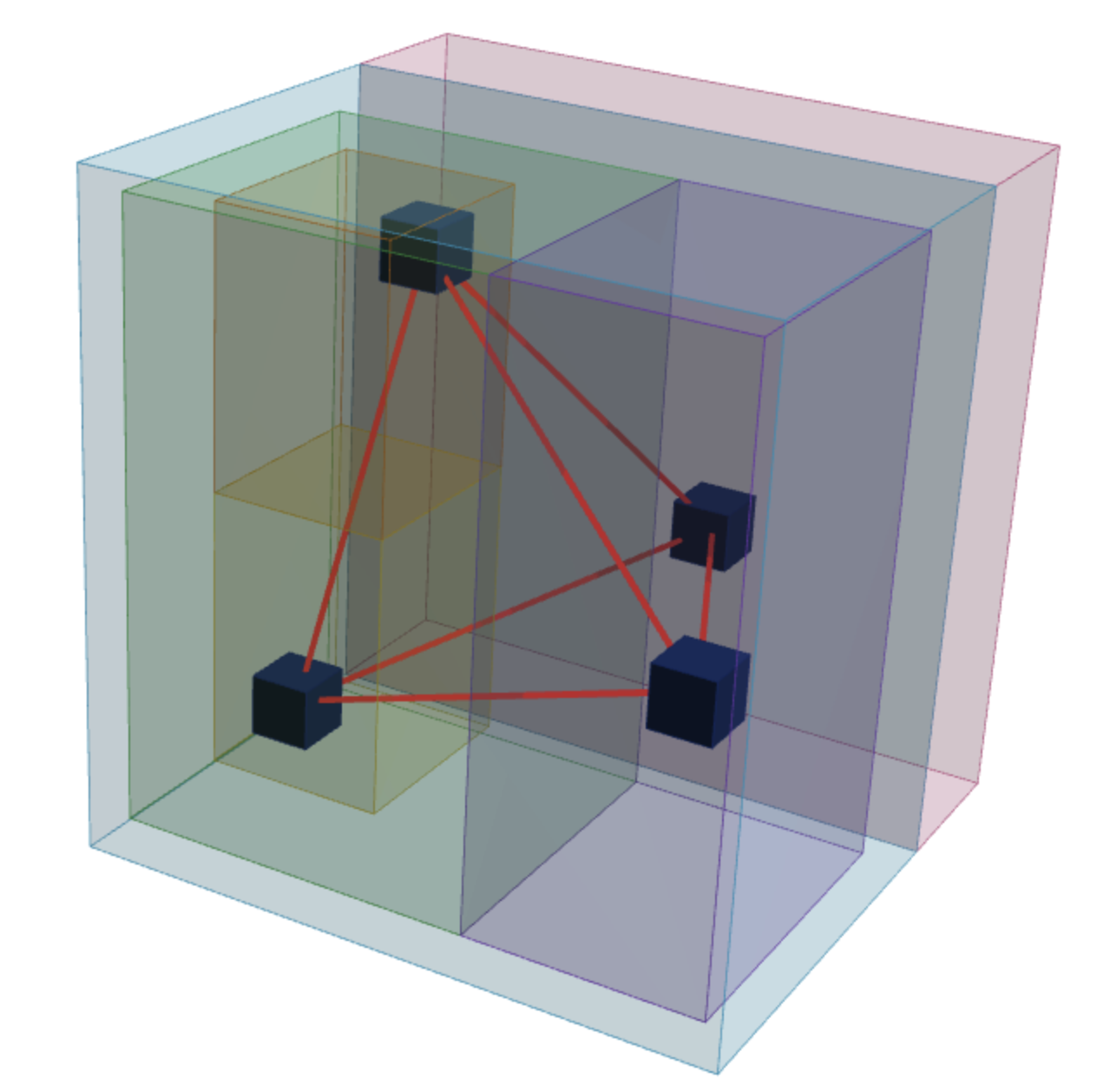}
        \caption{The decomposition of an intermediate region.}
        \label{fig: decomposition}
    \end{figure}
   
    Let us now discuss the construction of the map $f_{\cC}$. For concreteness, let us give the list of mergings we have done.
    \begin{enumerate}
        \item $A$ and $C$ merge to $R$,
        \item $R$ and $D$ merge to $Q$,
        \item $Q$ and $B$ merge to $U$,
    \end{enumerate}
    Notice that $U$ is a prism parallel to the surrounding cube $E$ and with all the cubes $A, B, C, D$ in its interior.
    
    Pick any of the above steps, where two prisms are being merged. Let us call them $X$ and $Y$. Let us call $Z$ the rectangular prism that is being used to merge them (for example, if $X = R$ and $Y = D$, then $Z = Q$). Its faces are parallel to those of $X$ and $Y$ and that strictly contains them in its interior. Then divide $Z$ by a parallel subface (i.e. the \say{membrane}), disjoint from $X$ and $Y$, into two smaller rectangles $Z_X$ and $Z_Y$ such that $X\subset Z_X$ and $Y\subset Z_Y$. Notice that we can always pick the membrane so that it only intersects the edges between $X$ and $Y$ but no others. 
    
    Define 
    \begin{align*}
        \mathcal{S}_{X, Z_X} &:= \overline{Z_X - X},\\ \mathcal{S}_{Y, Z_Y} &:= \overline{Z_Y - Y}.
    \end{align*}  
    By proposition \ref{prop: PP/LP existance of phi} there exist maps
    \begin{align*}
        \Phi_X: \mathcal{S}_{X, Z_X}\longrightarrow [a, b],\\
        \Phi_Y: \mathcal{S}_{Y, Z_Y}\longrightarrow [c, d],
    \end{align*}
    such that the level sets are rectangular prisms with faces parallels to those of $X$ and $Y$. We now identify $b$ and $d$ to create a path $P$ whose vertices are $a, b=d, c$. We define
    \begin{equation*}
        \Phi_{XY}:\mathcal{S}_{X, Z_X}\cup \mathcal{S}_{Y, Z_Y}\longrightarrow P,
    \end{equation*}
    as $\Phi_X$ on $\mathcal{S}_{X, Z_X}$ and as $\Phi_Y$ on $\mathcal{S}_{Y, Z_Y}$. The gluing lemma implies this is a well-defined continuous map as $\Phi_X$ and $\Phi_Y$ coincide on the membrane (whose image is $b = d$). Notice that
    \begin{equation*}
        \mathcal{S}_{X, Y} := \overline{Z - X - Y} = \mathcal{S}_{X, Z_X}\cup \mathcal{S}_{Y, Z_Y}.
    \end{equation*}
    Consequently, we have built a map 
    \begin{equation*}
        \Phi_{XY}:\mathcal{S}_{X, Y}\longrightarrow P,
    \end{equation*}
    By construction, the following holds
    \begin{align*}
        \Phi^{-1}_{XY}(a) &= X,\\
        \Phi^{-1}_{XY}(c) &= Y,\\
        \Phi^{-1}_{XY}(b = d) &= Z_X \cup Z_Y.
    \end{align*}
    Notice that this last preimage is a double bubble, as it is $Z$ together with the membrane. 

    Let us construct a map corresponding to $U$ and $E$. By proposition \ref{prop: PP/LP existance of phi}, there exists
    \begin{equation*}
        \Phi_{U, E}:\mathcal{S}_{U, E}\longrightarrow [e, f],
    \end{equation*}
     such that its level sets are prisms with faces parallel to those of $U$ and $E$. In here we have defined
    \begin{equation*}
        \mathcal{S}_{U, E} = \overline{E - U}.
    \end{equation*}

    We now have four maps $\Phi_{A, C}$, $\Phi_{R, D}$, $\Phi_{Q, B}$ and $\Phi_{U, E}$. Each goes to a different path. Notice that we can normalize notation as follows:
    \begin{align*}
        \Phi_{A, C}:& \mathcal{S}_{A, C} \longrightarrow [v_A, v_C],\\
        \Phi_{R, D}:& \mathcal{S}_{R, D} \longrightarrow [v_R, v_D],\\
        \Phi_{Q, B}:& \mathcal{S}_{Q, B} \longrightarrow [v_Q, v_B],\\
        \Phi_{U, E}:& \mathcal{S}_{U, E} \longrightarrow [v_U, v_E].
    \end{align*}
     For the first three maps, the edge on the codomain has an interior vertex whose preimage is a double bubble. Respectively, we call this interior vertex $w_R, w_Q, w_U$. The corresponding double bubble is the one whose exterior prism is $R, Q$ and $U$, respectively. 

    We intend to use the gluing map to construct out of these four maps a map on the intermediate region $\cC$. We have 
    \begin{equation}
    \label{eq: closed decomposition}
        \cC = \overline{E - (A \cup B \cup C \cup D)} = \mathcal{S}_{A, C} \cup \mathcal{S}_{R, D} \cup \mathcal{S}_{Q, B} \cup \mathcal{S}_{U, E}.
    \end{equation}
    Consequently, if we define appropriate maps that coincide in the intersections, the gluing lemma concludes the task. The target codomain will be a trivalent tree $\cT_{\cC}$. We construct it as an identification space, constructed from the edges $[v_A, v_C]$, $[v_R, v_D]$, $[v_Q, v_B]$ and $[v_U, v_E]$, by the identifications
    \begin{equation*}
        w_R \sim v_R, \; w_Q \sim v_Q, \; w_U \sim v_U.
    \end{equation*}
    Notice that this is indeed a trivalent tree. Its leaves are $v_A, v_B, v_C$, $v_D$ and $v_E$ and its trivalent vertices are $w_R = v_R$, $w_Q = v_Q$ and $w_U = v_U$. This verifies property $3$ about $\cT_{\cC}$.

    The nonempty double intersections among the closed sets of equation \eqref{eq: closed decomposition} are
    \begin{align*}
        \mathcal{S}_{A, C} \cap \mathcal{S}_{R, D} &= R,\\
        \mathcal{S}_{R, D} \cap \mathcal{S}_{Q, B} &= Q,\\
        \mathcal{S}_{Q, B} \cap \mathcal{S}_{U, E} &= U.
    \end{align*}
    Notice that the membranes, other than their edges, are not part of the intersections. The corresponding maps coincide in these intersections because their images are the identified vertices. Hence, by the gluing lemma, there is a continuous map 
    \begin{equation*}
        f_{\cC}: \cC \longrightarrow \cT_{\cC}
    \end{equation*}

    We claim that this map satisfies all the desired properties $1, 2$ and $4$. The trivalent vertices $v_R, v_Q, v_U$ satisfy
    \begin{align*}
        f_{\cC}^{-1}(v_R) &= \Phi_{A, C}^{-1}(v_R) = R_A \cup R_C,\\
        f_{\cC}^{-1}(v_Q) &= \Phi_{R, D}^{-1}(v_Q) = Q_R \cup Q_D,\\
        f_{\cC}^{-1}(v_U) &= \Phi_{Q, B}^{-1}(v_U) = U_Q \cup U_B,
    \end{align*}
    each of which is a double bubble. Notice this double bubble never intersects vertices, as there are none in $\cC$. Furthermore, by construction the intersections with the edges is transversal. This proves property $1$. (Notice that $R_A$, $R_C$, etc. are what we called $Z_X$ and $Z_Y$ when we constructed these maps above.)
    
    Property $2$ follows from proposition \ref{prop: PP/LP existance of phi},
    because these preimages are precisely the preimages of interpolations. These preimages are prisms whose faces are parallel to those of the cubes. As we know, by construction, the edges of the $1$-skeleton are never parallel to these faces. Furthermore, also by construction, if one of these spheres intersects an edge, then it has a point of the edge inside of it. Therefore,  all intersections are transversal (as opposed to tangent). 
    
    Finally, to prove property $4$ we recall proposition \ref{prop: bound on cubes}. 
    We already proved that for each interior points $x$, $f^{-1}_{\cC}(x)$ is a rectangular prism with faces parallel to those of the cubes. This also holds for the leaves themselves, where we recover$\partial A, \partial B, \partial C$ and $\partial D$ as preimages. However, we know more: this prism is a level set of one of the four glued maps. Hence, in order to prove this property, we will trace these maps in order.

    Recall that $\cC \cap K$ are the edges between the cubes and those segments from $A, B, C$ and $D$ that go out of $E$. Because the intersection is transversal, each of these segments intersects $f^{-1}_{\cC}(x)$  in a finite number of points. Suppose that $f^{-1}_{\cC}(x)$ interpolates between $X$ and $Z_X$. We agree to take $Z_X = E$ if $X = U$. Notice that for every $x$, interior to the edges, we can find an $X$ that satisfies this. 

    Let $p\in f^{-1}_{\cC}(x)\cap K$. Consequently,
     \begin{equation*}
         p\in K \cap \mathcal{S}_{X, Z_X}.
     \end{equation*}
    This intersection consists of closed intervals. By construction, all of these intervals have at least one of their endpoints on $\partial X$. Thus $p$ is connected to $\partial X$ through one of these edges. Hence, we have a correspondence
    \begin{equation*}
       f^{-1}_{\cC}(x)\cap K \longrightarrow \partial X\cap K.
    \end{equation*}
    We claim this correspondence is bijective. If $p_1$ and $p_2$ correspond to the same point, then $p_1p_2$ would be part of the given segment. However, $f^{-1}_{\cC}(x)$ is a rectangular prism with one of the endpoints of this segment inside. As this segment is transversal to the faces of the prism,  they can only intersect it in at most one point. Hence $p_1 = p_2$. 
    It is surjective for a similar reason. Given a point $p\in K\cap \partial X$, it must be part of a unique edge of $K$. The other endpoint of this edge lies on a vertex within another cube. By construction $Z_X$ only contains the cubes inside of $X$ (possibly $X$ itself). Hence, this edge has a point inside $f^{-1}_{\cC}(x)$ and one outside. It must therefore intersect it at some point. 
     
     The count of these points is, by definition, the weight of the spheres. Hence, we have proven is
     \begin{equation}
     \label{eqn: constancy of weight}
         w(f^{-1}(x), K) = w(\partial X, K).
     \end{equation}
    If $x_1$ and $x_2$ belong to the same edge, but are not trivalent vertices, they were not identified in the process of creating $\cT$ with another point. Hence, ther exists $X, Y$ such that 
    \begin{equation*}
        x_1, x_2 \in [v_X, v_Y]
    \end{equation*}
    In the case where there is an identification point between $v_X$ and $v_Y$, we must have $x_1$ and $x_2$ are on the same side of it. Otherwise they would be in different edge. This implies, by construction, that, $f^{-1}(x_1)$ and $f^{-1}(x_2)$ are interpolations between $X$ and $Z_X$ for the same $X$. Equation \eqref{eqn: constancy of weight} implies
    \begin{equation*}
        w(f^{-1}(x_1), K) = w(\partial X, K) = w(f^{-1}(x_2), K),
    \end{equation*}
    as desired. This concludes the proof. 
\end{proof}

\begin{proposition}
    \label{prop: wieght bound 6}
Let $\cC$ be an intermediate region and $f_{\cC}:\cC\longrightarrow \cT$ the map constructed in lemma \ref{lemma: existence of intermediate map}. Then for all $x\in\cT$ that is not a trivalent vertex,
\begin{equation*}
    w(f^{-1}(x), K) \le 6.
\end{equation*}
\end{proposition}
\begin{proof}
For an interior cube $V$ (i.e. $V$ is one of $A, B, C$ or $D$) define
\begin{equation*}
    \delta_V = 
\begin{cases}
  1, & \text{if } K \text{ enters or exits through } V,\\
  0,  & \text{otherwise }.
\end{cases}
\end{equation*}
Similarly, for two interior cubes $V$ and $W$ define
    \begin{equation*}
    \delta_{VW} = 
\begin{cases}
  1, & \text{if } K\text{ uses } VW,\\
  0,  & \text{otherwise }.
\end{cases}
\end{equation*}
Then we have
\begin{equation*}
    w(A, K) = \delta_A + \delta_{AB} + \delta_{AC} + \delta_{AD}. 
\end{equation*}
Similar for the other vertices.
When we merge two cubes, say $A$ and $B$, in a third prism $R$ the weight calculation is
\begin{equation*}
    w(R, K) = w(A, K) + w(B, K) - 2\delta_{AB} = \delta_A + \delta_B + \delta_{AC} + \delta_{AD} + \delta_{BC} + \delta_{BD}.
\end{equation*}
Similarly, merging now $R$ and $C$ into a prism $Q$, we get
\begin{equation*}
    w(Q, K) = \delta_A + \delta_B + +\delta_C + \delta_{AD} + \delta_{BD} + \delta_{CD}. 
\end{equation*}
And finally, of course, for the surrounding region
\begin{equation*}
    \delta(E, K) = \delta_A + \delta_B + \delta_C + \delta_D.
\end{equation*}
As all of these calculations at most sum six terms, each of which is a $0$ or a $1$, the bound of $6$ follows. 
\end{proof}

\begin{corollary}
\label{cor: miss gives 4}
    Let $\cC$ be an intermediate region and $f_{\cC}:\cC\longrightarrow \cT$ the map constructed in lemma \ref{lemma: existence of intermediate map}. 
    \begin{enumerate}
        \item Suppose there is one interior cube with one of its adjacent edges in $\cC$ not used by the knot, then for all $x\in \cT_{\cC}$ that are not trivalent vertices we have
    \begin{equation*}
        w(f^{-1}_{\cC}(x), K) \le 4.
    \end{equation*}

        \item On the other hand, if all edges are used in $\cC$, then for some $x\in\cT_{\cC}$ that is not a trivalent vertex, we have
        \begin{equation*}
            w(f^{-1}_{\cC}(x), K) = 6.
        \end{equation*}
    \end{enumerate} 
\end{corollary}
\begin{proof}
   By proposition \ref{prop: bound on cubes}, interior cubes must have either $0$, $2$ or $4$ points on their surfaces. Therefore, the initial weights are even. 
   
   The calculation of the successive weights implies they are also even, as you subtract \textit{twice} the number of shared edges. In particular, if one can guarantee the sum is not six, then it is at most four. 

   Only the two intermediate calculations, in the proof of proposition \ref{prop: wieght bound 6}, can be six. Concretely, with the notation of that proposition,
   \begin{align*}
       w(R, K) &= \delta_A + \delta_B + \delta_{AC} + \delta_{AD} + \delta_{BC} + \delta_{BD}\\
       w(Q, K) &= \delta_A + \delta_B + +\delta_C + \delta_{AD} + \delta_{BD} + \delta_{CD}
   \end{align*}
   Bot of this sums share $\delta_A$ and $\delta_{AD}$. Thus, declaring the cube in the hypothesis to be $A$ we see $\delta_A$ or $\delta_{AD} = 0$ (where, in the latter case, $D$ is the vertex on the other end of the missing edge.)

   We conclude these sums cannot be six, hence they are at most four. By property $4$ of lemma \ref{lemma: existence of intermediate map}, all the weights reduce the this count. This proves $1$.

   On the other hand, if such an intermediate region exists, then the count for $w(R, K) = 6$, because
   \begin{equation*}
       \delta_A = \delta_B = \delta_{AC} = \delta_{AD} = \delta_{BC} = \delta_{BD} = 1.
   \end{equation*}
    The result follows.
\end{proof}

\subsection{The vertex cubes}

The results of the previous subsection have taken care of almost all the regions where we need to construct the sphere decomposition. We are missing the exterior of the largest surrounding cube and the vertex cubes. We will return to the former in the proof of the main theorem.

\begin{lemma}
    \label{lem: vertex cubes map}
Let $C$ be a vertex cube around the vertex $V$. There exists a thickening of $V$ to a polyhedron $\mathcal{P}$ and a map $f_C: C \longrightarrow [0, 1]$ such that
\begin{enumerate}
    \item For each $0 < u \le 1$, the level set $f^{-1}(u)$ is a rectangular prism transversal or finitely tangent to $K$ thickened with $\mathcal{P}$.
    \item For each $0< u \le 1$, we have
    \begin{equation*}
        w(f^{-1}_C(u), K)\le 3,
    \end{equation*}
    where this weight is computed with respect to the thickened graph $K\cap C$ using $\mathcal{P}$.
    \item $f^{-1}_C(0)$ is a point disjoint from $K$.
\end{enumerate}
\end{lemma}
\begin{proof}
The vertex cubes have, by definition, the vertices of the $1$-skeleton that have been visited by the knot. Consequently, these vertices have degree $2$. 

For any interior point $p$ of $C$, we can create the linear homotopy from $C$ to $p$:
\begin{equation}
    H(x, u) = (1 - u)p + ux, \;\; 0\le u \le 1.
\end{equation}
This generic homotopy divides the cube into six different pyramids whose bases are parallel to a face of the cube. As long as $p$ is not collinear with $V$ and a point on the edges of the cube, the vertex $V$ will be in the \textit{interior} of one of these pyramids. This is an open condition and can always be achieved. We will assume we have fixed such a point $p$ and we will call it $p_V$. We can further assume that $p_V$ is not on $K$. We will denote the pyramid it lies by $\Delta$.

For each $0 < u \le 1$, the level set
\begin{equation*}
    C_u := \{H(x, u) \mid x\in C\},
\end{equation*}
is also a cube homothetic to $C$, from $p_V$. In particular, because it has $V$ in its interior, and the two edges incident at $V$ intersect $C$, they must also intersect $C_t$ at two points. 

Finally, because $V$ lies in the interior of 
$\Delta$, we can put a small convex polyhedron $\mathcal{P}$ around $P$ that lies in the \textit{interior} of $\Delta$. We further assume all the faces of $\mathcal{P}$ are transversal to the edges of the knot, which is to say, to the sides of the $1$-skeleton. We also assume its faces are transversal to the faces of the cubes. Together these are a finite number of linear conditions, so such convex polyhedron exists. We claim that if we thicken $V$ with this $\mathcal{P}$, the result will follow. 

Proposition \ref{prop: PP/LP existance of phi} applied to $p_V$ and $C$ implies the existence of a map 
\begin{equation*}
    f_C: C \longrightarrow [0, 1].
\end{equation*}
The level sets of these maps are the $C_u$ defined above (this is part of the proof of proposition \ref{prop: PP/LP existance of phi}. We refer the reader to it.) 

For the intersection with the edges, Property 1 follows from the corresponding property of proposition \ref{prop: PP/LP existance of phi}. For the intersection with the thickened vertex, we do as follows.
The thickened vertex $\mathcal{P}$ lies strictly inside $\Delta$. Hence, for a fixed level set, $\mathcal{P}$ can only intersect its face parallel to the base of $\Delta$. This face lies on a plane parallel to the faces of the cube. Consequently, the intersection will be a single component. Indeed, a plane and a \textit{convex} polyhedron intersect in a single component, which is the boundary of the convex polygon that the plane cuts of the solid polyhedron.
If this intersection is  a single point, then the intersection is finitely tangent. On the other hand, if the intersection is a proper polygon, then on each face of $\mathcal{P}$, the intersection is between two non-parallel planes. Hence, it is transversal. This concludes the proof of property $1$.

To see the validity of property $2$, notice that the thickened graph $K\cap \cC$ has three pieces: $\mathcal{P}$ and the two edges adjacent to $V$ without their segment interior to $\mathcal{P}$. 

The weight is, by definition, the number of connected components of the intersection. These intersections can be of two kinds: $f_C^{-1}(u)$ intersects only the edges or it intersects $\mathcal{P}$ and the edges.
Each edge, being a straight segment transversal to all the faces of the level set, can only intersect it at one point. Thus, the edges contribute to the weight at most $2$. We justified above that the only face of the cube that intersects $\mathcal{P}$ does it in at most one component (which might be an artifical intersection). We conclude that, in any case, the number of connected components in the intersection are at most $3$.

Finally, property $3$ is true because $f_C^{-1}(0) = p_V$ which by construction is the center of the homotopy and is disjoint from $K$ thickened by $\mathcal{P}$

This concludes the proof.
\end{proof}
\begin{remark}
    If we were uncareful and we just make an homotopy (or something similar) but we do not control where are the intersections with $\mathcal{P}$ we could land in the situation that all the faces intersect $\mathcal{P}$ and, on top, they intersect the edges outside. This would give a count of $8$, which would make our bound grow unnecessarily.
\end{remark}

\subsection{The main theorem}

We have now done all the required work. We can finally prove

\begin{theorem}
    \label{thm: sphere decomposition}
Let $K$ be a knot that can be embedded in a finite iteration of the Sierpinski tetrahedron. Then $K$ admits a sphere decomposition of weight at most $6$.
\end{theorem}
\begin{proof}
    Let $\cC_1,..., \cC_m$ be all the intermediate regions. Let $E$ be the exterior cube. Let $P_1,..., P_k$ be the vertex cubes used in the construction. In other words, inside of them are the vertices of the smallest tetrahedrons visited by the knot. Without loss of generality assume that $\cC_1$ is the intermediate region whose surrounding cube is $E$ itself.

    We can define an homotopy $H:E\times [0, \infty) \longrightarrow \overline{\mathbb{S}^3-E - \{\infty\}}$ by
    \begin{equation*}
        H(e, u) = ue
    \end{equation*}
    At a fixed time $u_0$, the image of $e \longrightarrow H(e, u_0)$ is a cube $E_{u_0}$ with sides parallel to $E$. Consequently, we have a map 
    \begin{equation*}
        f_E: \overline{\mathbb{S}^3-E} \longrightarrow [1, \infty],
    \end{equation*}
    given by 
    \begin{equation*}
        f_E(x) = t
    \end{equation*}
    if $x\in E_t$ and $f_E(\infty)=\infty$. We shall call $\infty$ the root of this tree (which is an edge) and call it $\cT_{\infty}$. This map is continuous and the proof is the same as that of proposition \ref{prop: PP/LP existance of phi}.

    In this way, we have a family of continuous maps $f_E, f_{C_1},..., f_{C_m}$, $f_{P_1},..., f_{P_k}$. For the intermediate regions, the maps $f_{C_j}$ are the ones constructed in lemma \ref{lemma: existence of intermediate map}. For the vertex cubes the maps are constructed in lemma \ref{lem: vertex cubes map}.

    Suppose $\cC$ is an intermediate region and $X_1,..., X_k$ are the contiguous regions it surrounds. We know $1\le k \le 4$. Each $X_i$ is a cube, possibly a vertex cube.
    
    By construction, the tree $\cT_{\cC}$ has a vertex labeled after each $X_i$. In the same way, the trees $\cT_{X_i}$ has its root labeled after $X_i$. Thus, if we identify the leaf associated to $X_i$ in $\cT_{\cC}$ with the root of $\cT_{X_i}$, for each $1\le i\le k$, we obtain a new tree $\cT_0$. The vertices where the gluing happened have degree two because each glued tree contributed one edge. In $\cT_0$ we now consider this gluing point as part of the edge and \textbf{not} a vertex. 
    
    Notice that
    \begin{equation*}
        \cC \cup X_1 \cup \cdots \cup X_k
    \end{equation*}
    is a finite union of closed sets and the nonempty intersections among them are the exterior boundaries of $X_1,..., X_k$. Consequently, we can define
    \begin{equation*}
        f_0: \cC \cup X_1 \cup \cdots \cup X_k \longrightarrow \cT_0,
    \end{equation*}
    as $f_{\cC}$ in ${\cC}$ and $f_{X_i}$ in $X_i$. The gluing lemma implies this map is continuous. Notice that under this map the preimage of the gluing points are precisely the boundaries of $X_i$, which are topological spheres.

    Let us summarize what have we done. For each intermediate region, we have created a new map that extends $f_{\cC}$ into the interiors of the regions it surrounds. We have done this in such a way that the codomain is a tree built out of the trees we already had. We can now proceed inductively from the first intermediate region towards the vertex cubes. Indeed, our previous construction guarantees we can extend the map up to $C_1 \cup X_1\cdots \cup X_k$. If all of these $X_1,..., X_k$ are vertex cubes then we have
    \begin{equation}
        E = \cC_1 \cup X_1 \cup \cdots \cup X_k,
    \end{equation}
    and we have built a map in all of $E$. If not then we repeat this argument in each of the $X_i$ that is an intermediate region. Once more the gluing lemma guarantees that the function we already have will glue continuously onto the tree formed by identifying the leaf associated to $X_i$ with the corresponding root. This process must finish because there is a finite number of regions. 

    Hence, we have a map $f_1: E\longrightarrow \cT_1$ that is built by successively using the gluing lemma to attach together all the $f_{C_1},..., f_{C_m}$, $f_{P_1},..., f_{P_k}$. Finally, we apply the gluing lemma one last time to glue $f_1$ and $f_E$ along the boundary of $E$, which is a cube, to produce a map $f:\mathbb{S}^3\longrightarrow \cT$. The tree $\cT$ is obtained by gluing the trees of $f_1$ and $f_E$  at their corresponding vertex associated to $E$, but this new point is not considered a vertex of the new tree. Our new tree has as root $\infty$.

    What now prove that $f:\mathbb{S}^3\longrightarrow \cT$ is a sphere decomposition. We verify the properties one by one.

    \underline{\textit{(SD 1) For all $x\in L(\cT)$, $f^{-1}(x)$ is a point disjoint of $K$:}} The trivalent tree $\cT$ is formed by gluing the trees
    \begin{equation*}
        \cT_{\infty}, \cT_{\cC_1}, \cT_{\cC_2},..., \cT_{\cC_m}, \cT_{P_1},..., \cT_{P_k}.
    \end{equation*}
    The leaves of $\cT$ are those leaves of the above trees that were not glued. To do not be glued, the leave must be associated to a region that has no further region inside of them. This only occurs for $P_1,..., P_k$ and $\overline{E^c}$ (whose inside is the region containing $\infty$). By construction, these trees are edges representing an homotopy from the region to a point. Neither of these touch $K$. Indeed, the former do not by lemma \ref{lem: vertex cubes map}. The latter, being $\infty$, is not on $K$. This proves the first property.

    \underline{\textit{(SD 2)For all $x\in V(\cT)-L(\cT)$ , $f^{-1}(x)$ is a double bubble transverse to $K$:}} The vertices in $V(\cT) - L(\cT)$ are the trivalent vertices. These vertices are not being identified to any other points, thus its preimage coincides with that of the corresponding map being glued. Furthermore, the only trees with trivalent vertices are those corresponding to  $\cC_1,..., \cC_m$. The result now follows by lemma \ref{lemma: existence of intermediate map}.

    \underline{\textit{(SD 3) For all $x\in \cT$ interior to an edge, $f^{-1}(x)$ is a sphere transverse or finitely tangent to $K$:}}
    The vertices interior to an edge are of two sorts: those that are points of identification but not considered vertices and those that are interior to an edge and were not identified. The former ones correspond to the regions where two closed sets intersected. By construction, these are all surfaces of cubes transversal to $K$.
    The latter points satisfy what we want by lemma \ref{lemma: existence of intermediate map} in the case of intermediate regions. For the case of $f_{\infty}$, all these regions are surfaces of cubes. Finally, for the case of $f_{P_k}$, this follows by lemma \ref{lem: vertex cubes map}.
    
The three properties hold. Hence, $f$ is a sphere decomposition. To verify its width, we realize that each interior point has as a preimage the surface of a sphere of which we have already computed its weight.

We begin by interior points of the glued trees. For $f_{\infty}$ the weight is zero because they never intersect the knot. For every intermediate region, proposition \ref{prop: wieght bound 6}, guarantees the intersection is at most $6$. For the vertex cubes, lemma \ref{lem: vertex cubes map} implies the weight is bounded by $3$. 

For the interior points of $\cT$ that were formed by identifying leaves of two other trees, notice that their preimage is precisely a boundary between an intermediate region and a contiguous one. For all of them, the intersection is at most $4$, by proposition \ref{prop: bound on cubes}. Consequently, the spherewidth of $f_{\cC}$ is bounded by $6$.

This concludes the proof of the theorem.
\end{proof}
\begin{remark}
\label{remark: ignore last edges}
    In the course of the proof we have seen that the number of leaves of $\cT$ is equal to the number of vertex cubes plus one, the root, associated to $\infty$. The edges associated to the homotopies for the vertex cubes or the outer cube have weight at most $3$. Their other vertex is one associated to the vertex cubes or, in the case of $\infty$, to the surrounding cube $E$. Consequently, the structure of the tree can be read from the rest of the information. In particular, in our further discussions we will ignore these edges as their information do not change our conclusions.
\end{remark}

We finally have our desired result.

\begin{theorem}
    \label{thm: Not all knots are in}
    Let $T(p, q)$ be a torus knot with $\min\{p, q\} > 9$, then
    $T(p, q)$ cannot be embedded in the $1$-skeleton of a finite iteration of the Sierpinski tetrahedron. In particular, there are infinitely many knots that cannot be embedded in the $1$-skeletons of the finite iterations of the Sierpinski tetrahedron.
\end{theorem}
\begin{proof}
In \cite[Corollary 1.3]{LunelSphereWidth} it is proven that
\begin{equation*}
    sw(T(p, q))\ge \dfrac{2}{3}\min\{p, q\}.
\end{equation*}
If $T(p, q)$ is in a finite iteration of the Sierpinski tetrahedron, theorem \ref{thm: sphere decomposition} implies
\begin{equation*}
    \dfrac{2}{3}\min\{p, q\} \le 6.
\end{equation*}
In other words,
\begin{equation*}
    \min\{p, q\} \le 9.
\end{equation*}
There are infinitely many torus knots not satisfying the above. This concludes the proof.
\end{proof}

\section{An example with the Trefoil}
\label{sec: An example with the Trefoil}

In this section we will give an example for the trefoil $3_1$. Part of our goal is to convince ourselves that the sphere decomposition, and hence the tree, can be read in an straightforward way from the diagram in the combinatorial representation.

\subsection{Intermediate regions and the combinatorial representation}

We recall that a combinatorial representation is a \say{map} of a finite iteration that allows us to see a finite iteration in two dimensions. It is constructed by flattening each iteration. This sequence can be seen in figure \ref{fig: combinatorial representation}.

\begin{figure}[H]
    \centering
    \includegraphics[scale = 0.3]{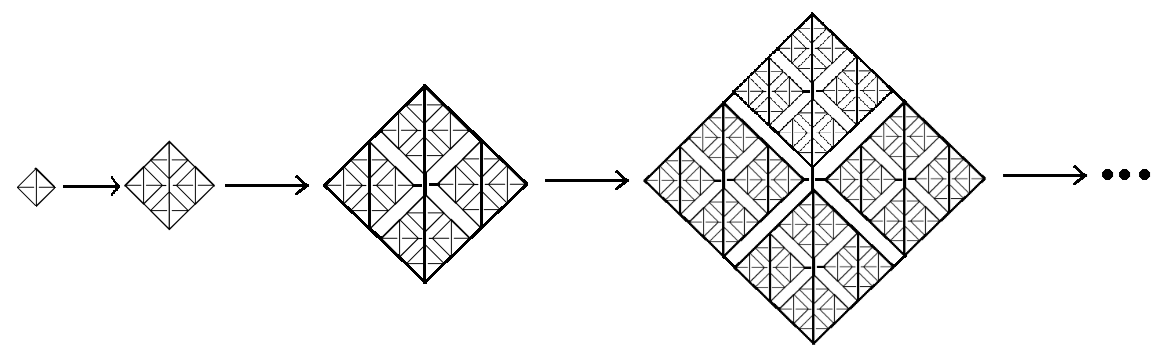}
    \caption{Sequence of  combinatorial representation}
    \label{fig: combinatorial representation}
\end{figure}

In particular, we can read the intermediate regions we used to construct our sphere decompositions, via the flattening of a single iteration. We see this in figure \ref{fig: regions and combinatorial representation}.

\begin{figure}[H]
    \centering
    \includegraphics[scale = 0.5]{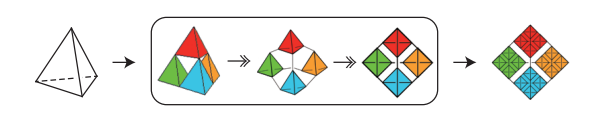}
    \caption{For a particular iteration, the four interior cubes correspond to the four color squares.}
    \label{fig: regions and combinatorial representation}
\end{figure}

To construct the tree and count the weights of the spheres, what matter is to know how to connect different regions and the number of components on each intermediate sphere. By localizing in the corresponding zone of the combinatorial representation, this can be readily counted.

\subsection{The example for the trefoil}

In \cite{knotsinsidefractals} we have found the trefoil in a finite iteration of the tetrahedron. This corresponds to the combinatorial representation show in figure \ref{fig: combinatorial representation for trefoil}.
\begin{figure}
    \centering
    \includegraphics[scale = 0.1]{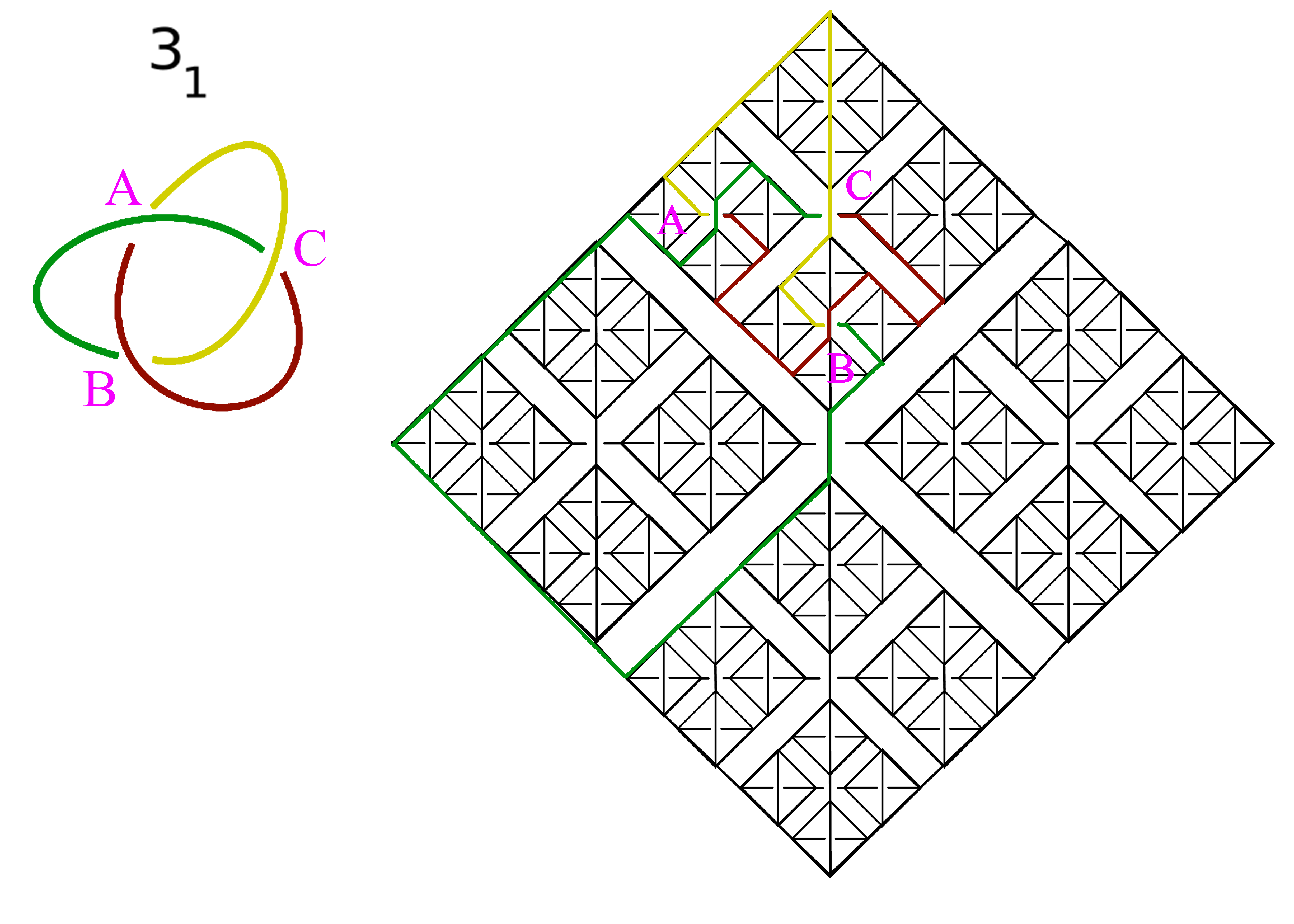}
    \caption{A combinatorial representation for $3_1$. The points $A$, $B$ and $C$ help localize the three crossings.}
    \label{fig: combinatorial representation for trefoil}
\end{figure}

In order to locate ourselves as we read, we will denote the different squares according to their position with respect to the north (N), south(S), east (E) and west (W) direction. For example, the first iteration has four pyramids corresponding to $N, S, E$ and $W$. For another example, take the smallest tetrahedron, in figure \ref{fig: combinatorial representation for trefoil}, where the letter \textbf{B} is located. Its location is $NSS$. We are in the third iteration so we need at most three coordinates. Finally, we shall call $F$ the exterior cube surrounding the whole tetrahedron. 

Following our convention, the different intermediate regions, \textit{besides} those in the smallest cubes, are:

\begin{enumerate}
    \item $F$ with interior cubes $N$, $S$ and $W$,
    \item $N$ with interior cubes $NN$, $NS$, $NW$, $NE$.,
    \item $NW$ with interior cubes $NWN$, $NW$, $NWW$ and $NWE$,
    \item $NS$ with interior cubes $NSN$, $NSW$, $NSS$ and $NSE$. 
\end{enumerate}

The smallest cubes of our configuration are: $W, S, NN, NE$, $NWN$, $NWS$, $NWW$, $NWE$, $NSN$, $NSW$, $NSS$ and $NSE$. Thus our final tree $\cT$ will have $12$ leaves besides the root corresponding to $\infty$. It will be constructed by gluing four trees, one per intermediate region. As per remark \ref{remark: ignore last edges}, we do not put the edges connecting to the leaves.

In order to see this more clearly, and be able to read our diagram easily, we have filled in the smallest cubes and highlighted the connections in figure \ref{fig: Diagram from representation}. Each intermediate region is filled in with a different color. The connection edges within regions that are used by the knot appear in blue. Those that do not appear in red. Notice that $E$ does not appear because the knot does not visit that pyramid.

\begin{figure}[H]
    \centering
    \includegraphics[scale = 0.4]{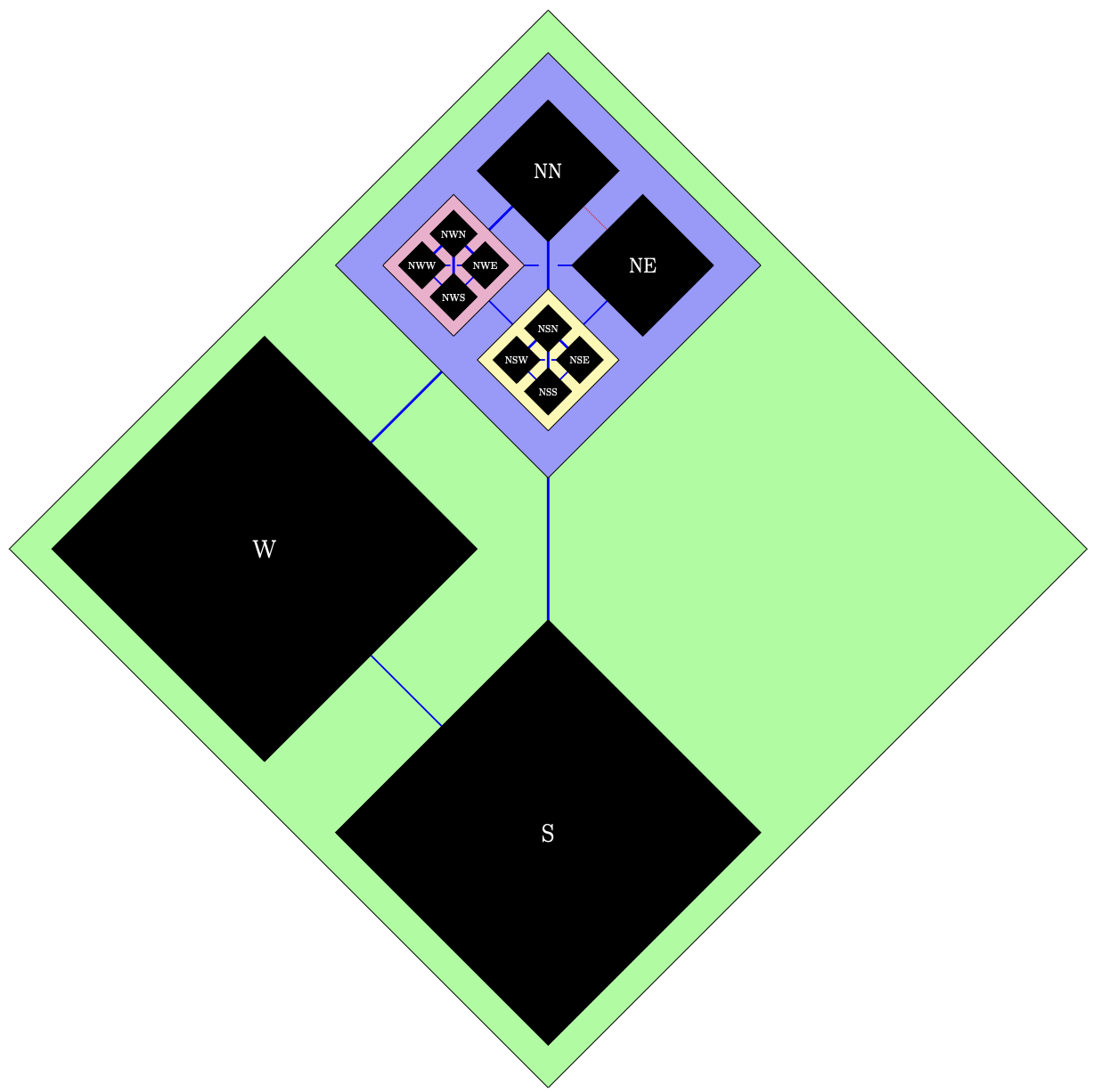}
    \caption{All the information to build the tree.}
    \label{fig: Diagram from representation}
\end{figure}

The way we will merge each of the intermediate regions, given in our above list, is as follows:
\begin{enumerate}
    \item We merge $N$ and $W$. Then we merge with $S$. 

    \item We merge $NN$ with $NW$. Then, we merge the result with $NS$. Finally, we merge with $NE$.

    \item We merge $NWW$ with $NWS$. We then merge the result with $NWE$. Finally, we merge with $NWN$.

    \item We merge $NSW$ with $NSS$. We then merge the result with $NSE$. Finally, we merge with $NSN$.
\end{enumerate}

Finally we must count the number of components. We have repeatedly seen this is equivalent to counting how many points are on the surface, although the former count is half the latter one. This is done by inspection. We look at each smallest regions and count. For example $W$ has a single component. $S$ has a single component. On the other hand, $NSE$ has two components. 

By \ref{lemma: existence of intermediate map}, the weight of the spheres, does not change along edges. Thus, we only need to understand the components in tetrahedrons and not at intermediate stages. 

We can now see the tree. We have them highlighted by intermediate region.  Finally, the number registered is the number of components, thus the weight is double that number. We emphasize this is because, as we have seen in the course of the proof of theorem \ref{thm: sphere decomposition}, all the involved intersections in the intermediate regions are transversal. Hence every connected component goes in and goes out at different points.

\begin{figure}[H]
    \centering
    \includegraphics[width=0.5\linewidth]{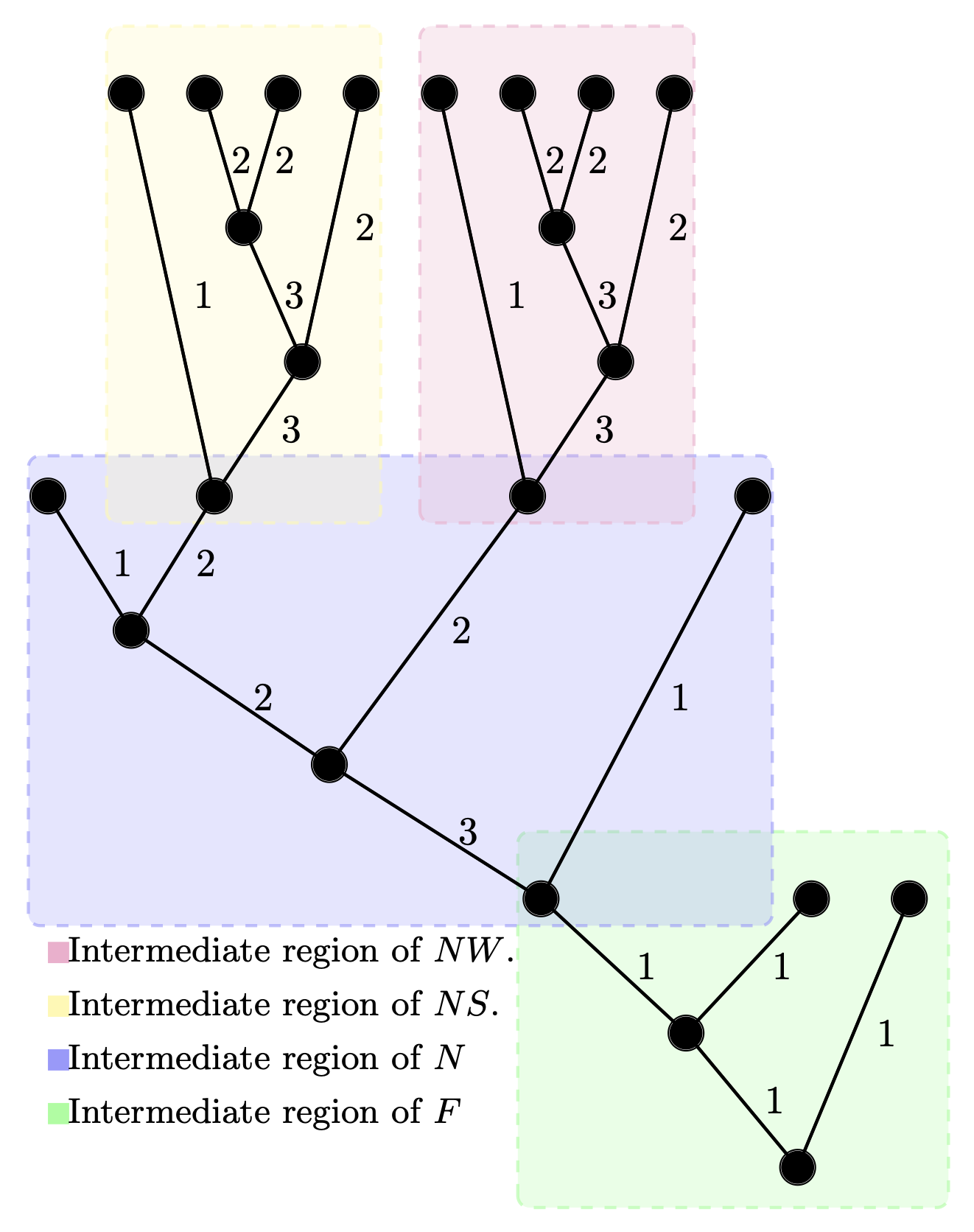}
    \caption{Trivalent tree for the trefoil.}
    \label{fig: Trivalent tree}
\end{figure}

We conclude this is a sphere decomposition with width $6$. Consequently $sw(3_1)\le 6$.

\begin{remark}
    We did not need to construct the whole tree to deduce the spherewidth of $f$. The intermediate regions at $NW$ and $NS$ use \textit{all} edges among the interior cubes and exits. Hence, corollary \ref{cor: miss gives 4} implies  $sw(f) = 6$.

    This doesn't imply that $sw(3_1) = 6$. As we will see, there are other sphere decompositions that bring the spherewidth lower.
\end{remark}

\section{Further Questions and Concluding Remarks}
\label{sec: Further Questions and Concluding Remarks}

\subsection{Pretzel Knots}

In \cite{knotsinsidefractals} we have proven that all Pretzel knots can be found in the finite iterations of the Sierpinski tetrahedron. Theorem \ref{thm: sphere decomposition} then implies its spherewidth is bounded by $6$. This bound can be improved if one forgets about the tetrahedron.

\begin{proposition}
\label{prop: pretzel bound 4}
Let $P(a_1,..., a_n)$ be a Pretzel knot, then
    \begin{equation}
    \label{eq: pretzel bound 4}
        sw(P(a_1,..., a_n)) \le 4.
    \end{equation}
\end{proposition}
 \begin{proof}
      We will just describe the main steps as the details are exactly the same as those of our previous results. We will do it for pretzels of three strands for ease of explanation.

      The proof consists of three main steps:
      \begin{enumerate}
          \item Get the appropriate thickened embedding.
          \item Find the regions to be isolated so that the combinatorics outside of them are doable.
          \item Fill in with what happens in the isolated regions with their own computations.
      \end{enumerate}

      For pretzel knots step $1$ can be done. We can assume all edges used are not parallel to any of the coordinate planes. We can thicken each vertex such that the weight in there, as we sweep, does not exceed $4$. We will leave the details of this to the interested reader.

      The regions that we will isolate are the individual tangles as well as the upper and lower strand joining the first and last tangles. To isolate them, we put around them solid prisms that are disjoint and that can be separated by planes in the $X$ and $Y$ directions. In this case we do not need the third direction because the membranes will only be parallel to two of the coordinate planes.

      This step is analogous to our intermediate regions. We knew something happened on beyond the boundary, whether toward the exterior cube or toward the vertex cubes. However, that was unnecessary for the construction of the sphere decomposition in this region and for the calculations of the weights in there.

      \begin{figure}[H]
    \centering
    \begin{subfigure}{0.28\textwidth}
        \centering
        \includegraphics[width=\linewidth]{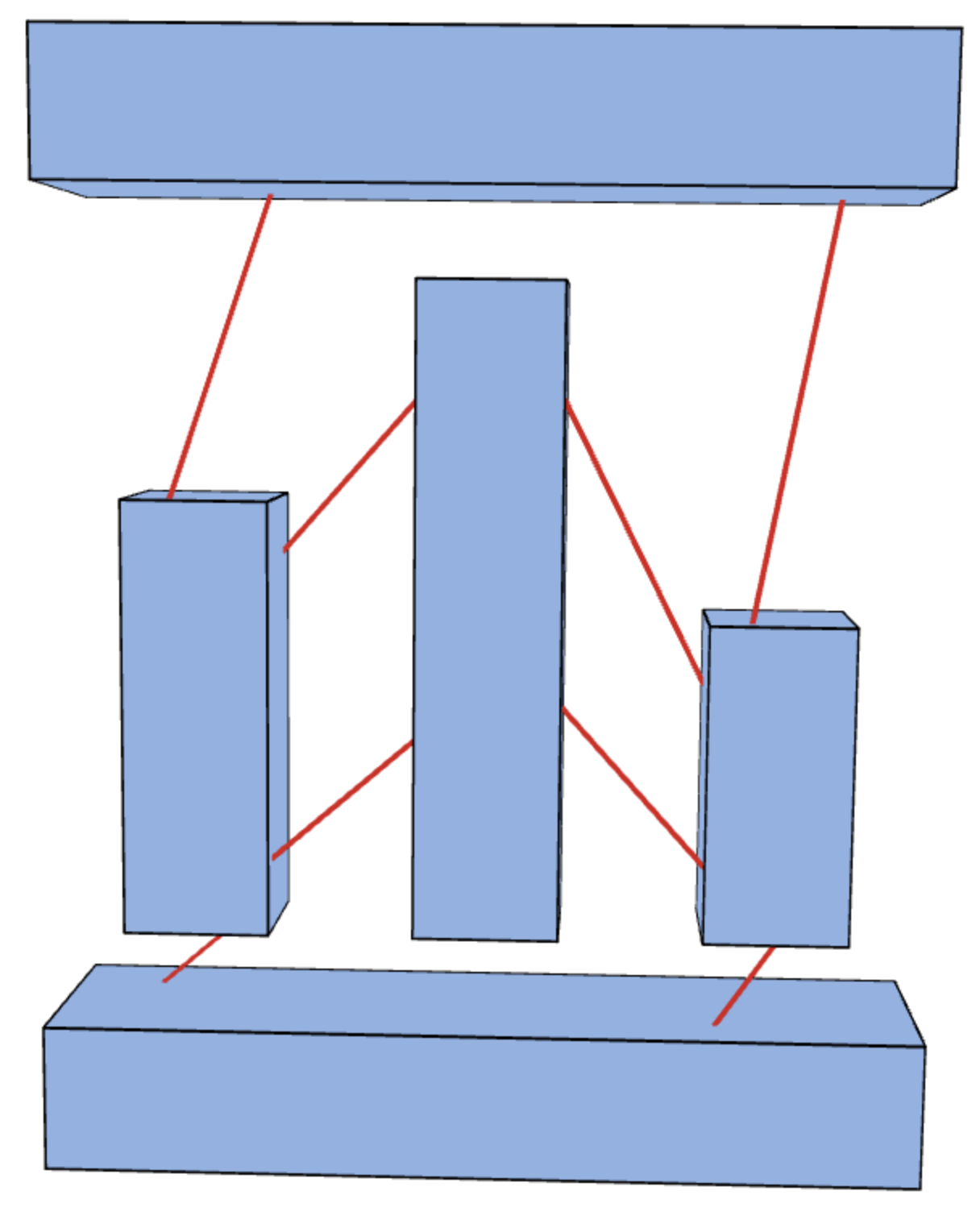}
        \caption{The isolated regions.}
        \label{fig: isolated}
    \end{subfigure}
    \hfill
    \begin{subfigure}{0.28\textwidth}
        \centering
        \includegraphics[width=\linewidth]{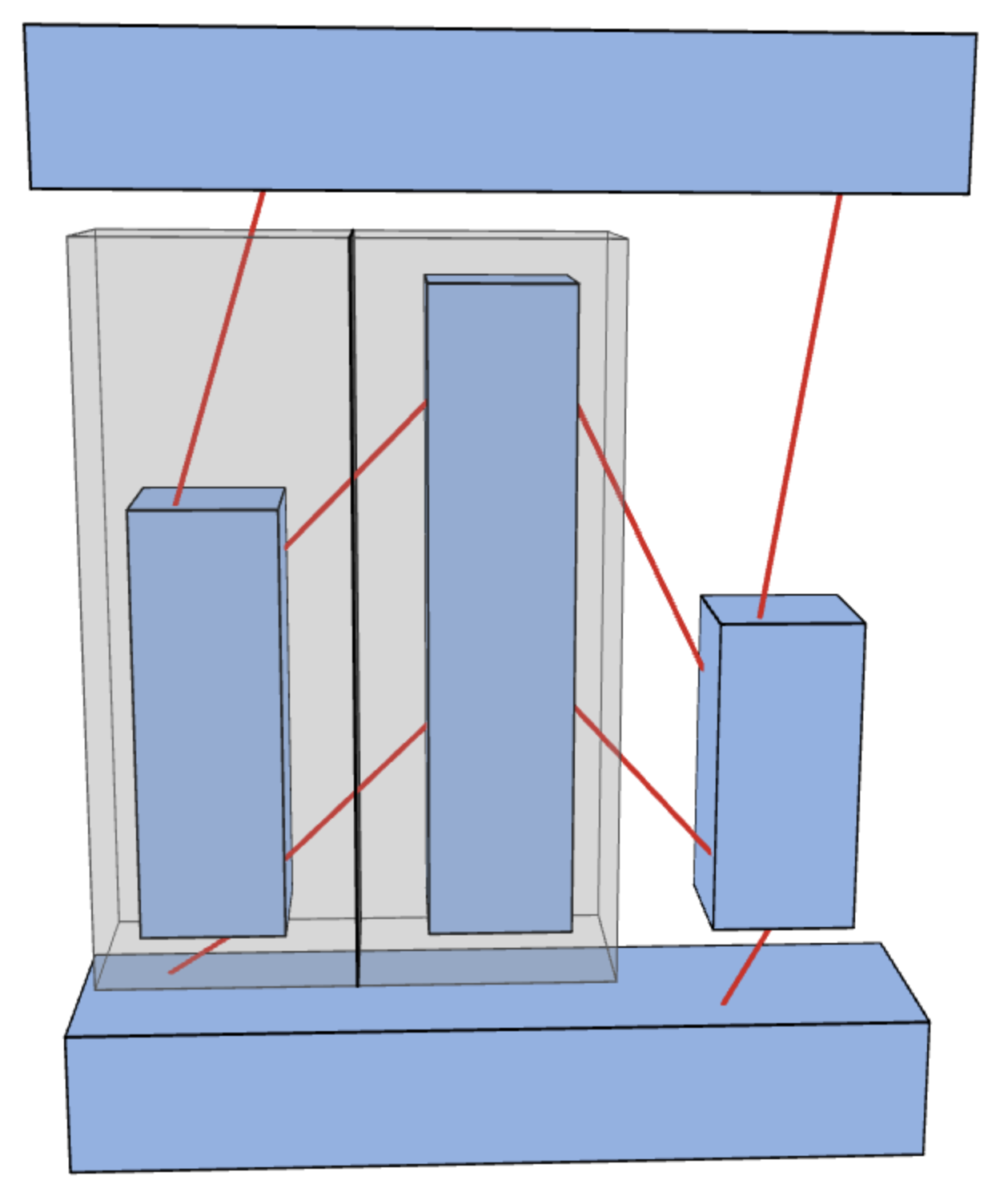}
        \caption{The first merge.}
        \label{fig: first merge}
    \end{subfigure}
    \hfill
    \begin{subfigure}{0.28\textwidth}
        \centering
        \includegraphics[width=\linewidth]{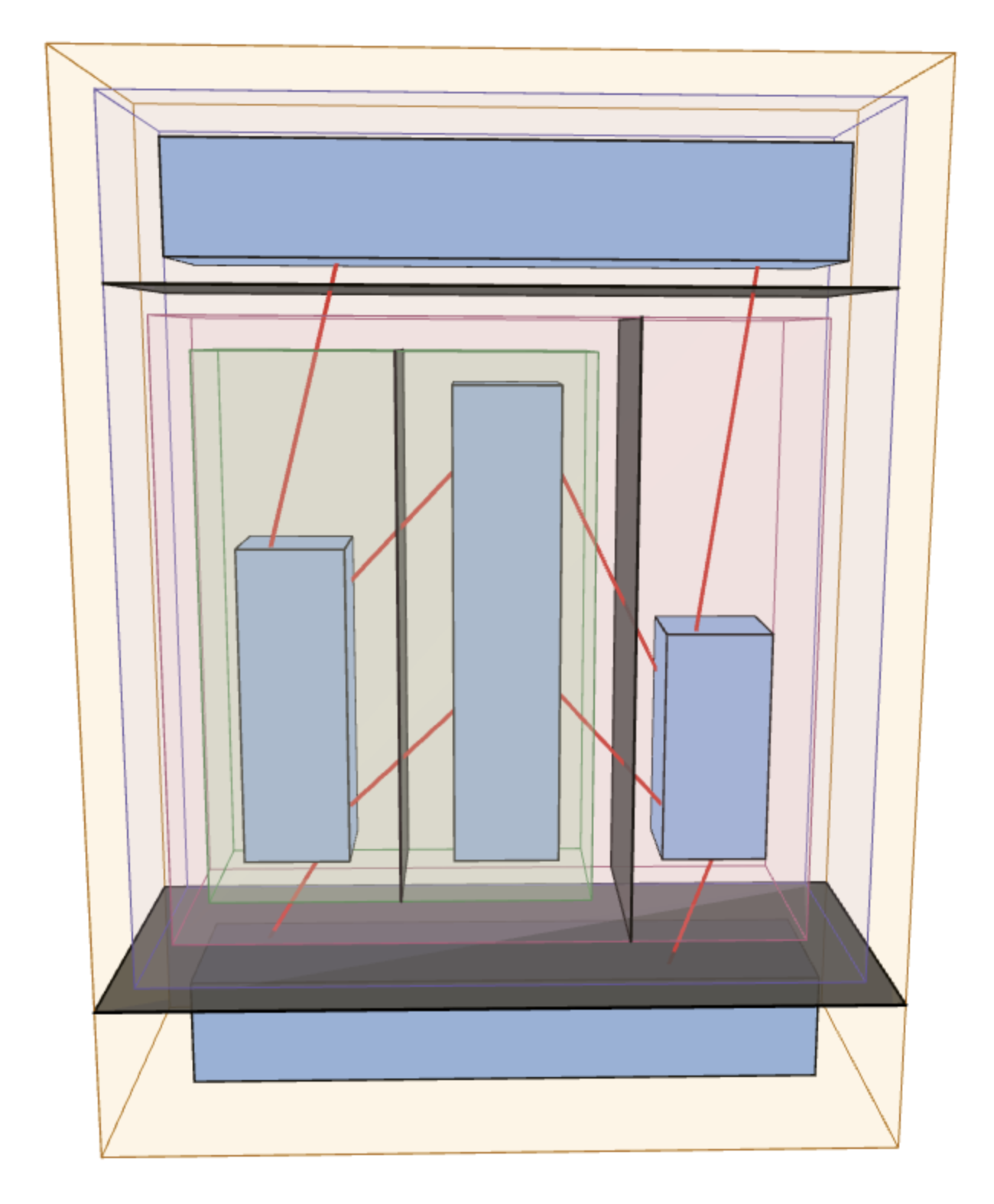}
        \caption{All the regions merged.}
        \label{fig: all regions merged}
    \end{subfigure}
    \caption{Process of merging with double bubbles}
    \label{fig: merging process pretzel}
\end{figure}

    We progressively \say{merge} the different isolated regions. This can be seen in figure \ref{fig: merging process pretzel}. Notice that figure \ref{fig: all regions merged} is the corresponding figure to figure \ref{fig: decomposition}.

    Just as in lemma \ref{lemma: existence of intermediate map}, a map can be created that glues all the interpolations on each merging process. Furthermore, the same analysis proves the weights depend only on the boundaries of these interpolations and does not change through them. Hence, in the above, we only need to verify what happens on them. An easy calculation shows that all of them have weight $4$ or $2$. This can be visually confirmed bt seing the intersections of the edges with the boundaries in figure \ref{fig: all regions merged}.

    This doesn't conclude our task. We must see that inside each isolated region the bounds are also preserved. For each one of the tangles, we can pick a point very close to the bottom face and make the homotopy to the surrounding prisms. One can assume, that except for the strands that go out of the prism, the tangle is entirely contained in the pyramid parallel to the upper face. This guarantees that the tangle is not cut into many different parts by different sides of the cube. This would bring the weight up. In this way, every level set cuts the tangle in three parts: the upper face cuts it when it enters, then it cuts it again when it goes out. However, it goes out through the paths going out of the surrounding prism. Hence, there are four cuts. We can see an example of this in figure \ref{fig: Inside a tangle prism}.

    \begin{figure}[H]
        \centering
        \includegraphics[width=0.3\linewidth]{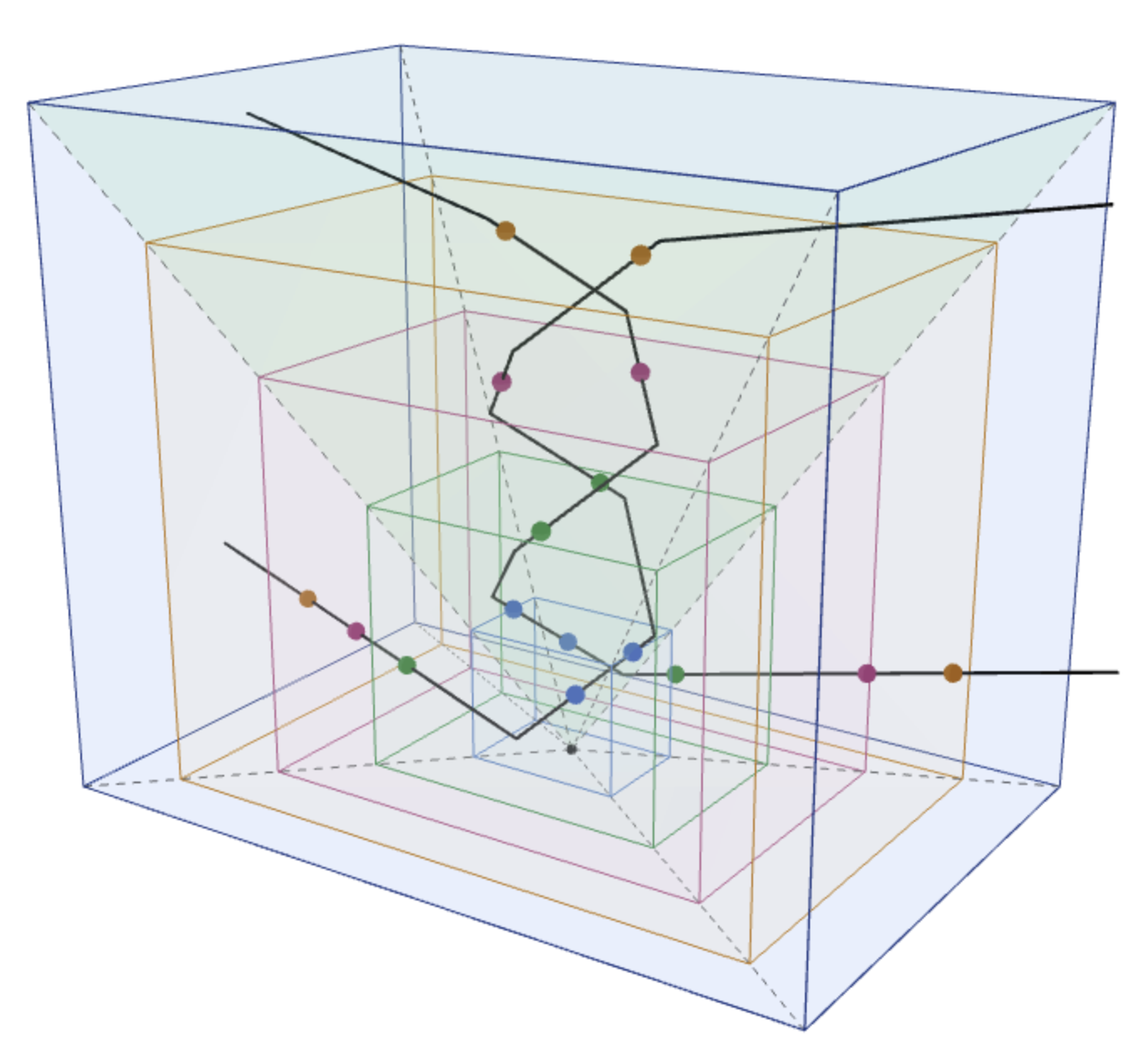}
        \caption{Inside a tangle prism. Intersections of the same level sert are marked with the corresponding color.}
        \label{fig: Inside a tangle prism}
    \end{figure}

    In here, it is important to take good care of the thickening of the vertices. If we do not isolate them, then as we sweep we might intersect many at the same time. If several of them have artificial intersections they will make the count go up. This can be done, for example, with separating polyhedrons.
    A similar analysis works for the other two prisms (the horizontal ones) where the count is bonded by $2$.
     
    This concludes our sketch of proof.
 \end{proof} 
 \begin{remark}
     That Pretzel knots have spherewidth bounded above by $4$ is probably well known. However, I searched online and with A.I. and it did not find explicit sources. It just repeatedly send me to \cite[Figure 3]{LunelSphereWidth}. 
 \end{remark}

Proposition \ref{prop: pretzel bound 4} improves our bound $sw(3_1)\le 6$. Indeed, $3_1 = P(1, 1, 1)$, and thus $sw(3_1)\le 4$.
This motivates the following:

\begin{question}
    Is there an embedding of the Pretzel knots in the finite iterations of the Sierpinski tetrahedron and a corresponding sphere decomposition that recovers the bound $4$? 
\end{question}

This question seems dual to the upper bound of $6$ that we have found. Concretely, the finite connectivity between different tetrahedrons force the spherewidth to do not grow too much. However, it seems that at the same time, this same scarcity of options to change of tetrahedron forces the knot to visit pyramids for \say{superficial} purposes, creating intersections with the level set spheres that otherwise could have been avoided.

\subsection{What other knots are there in the finite iterations?}

In \cite[Theorem II]{Kawauchi} it is proven that the only torus knots that are Pretzel knots are
\begin{align*}
    T(2, n) &= P(1, ..., 1),\\
    T(3, 4) &= P(3, 3, -2) = 8_{19},\\
    T(3, 5) &= P(5, 3, -2) = 10_{124}.
\end{align*}
Notice all these torus knots indeed satisfy the conditions of theorem \ref{thm: Not all knots are in}. There are many torus knots not in the above list, with $\min{p, q}\le 9$, and that are not Pretzel. 

Ultimately, this leads to the following.

\begin{question}
    Can one describe in some sensible way what are the knots that appear in the finite iterations of the Sierpinski tetrahedron? This could entail a (hopefully?) finite list of families of knots with some understood exceptions. 
\end{question}

Simpler versions of this question deal with finite concrete families. For example, is the family $T(3, n)$ for $n$ not a multiple of $3$, in there? The first two members are because they are Pretzel knots. 

Similarly, there are also knots that are neither Pretzel nor torus knots. The smallest prime example is $8_{21}$. One of its knot diagram, taken from The Knot Atlas, is see in figure \ref{fig: Diagram 821}.
\begin{figure}[H]
    \centering
    \includegraphics[scale = 0.4]{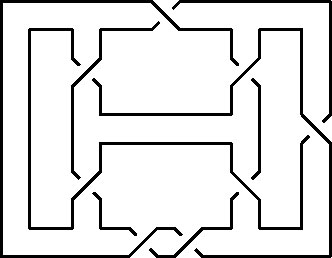}
    \caption{Knot Diagram of $8_{21}$}
    \label{fig: Diagram 821}
\end{figure}

The diagram of figure \ref{fig: Diagram 821} is very reminiscent of the combinatorial representations. At the moment of writing, I do not know if one can find it in some iteration.

In any case, more than torus knot families and Pretzel knots might be needed. What other family could be there lurking in the finite iterations?

\appendix
\section{Geometric Lemmas on Polyhedrons}
\label{appx: Geometric Lemmas on Polyhedrons}

\subsection{Interpolations}

We will review several facts about interpolations of quadrilaterals in $\mathbb{R}^3$. This results, or similar, are probably well known. However, we review them to highlight the aspects that we will require in our constructions.

\begin{definition}
Given two points $A$ and $a$ in $\mathbb{R}^3$, the \textbf{interpolation} between them is the continuous map $\alpha:[0, 1]\longrightarrow\mathbb{R}^3$ given by
\begin{equation*}
    \alpha(t) = (1 - t)A + ta.
\end{equation*}
\end{definition}

The interpolation is a parametrization of the segment $Aa$ from $A$ to $a$. The parameter $t$ represents the proportion $A\alpha(t)/Aa$. Consequently, as long as $A$ and $a$ are different, this map is injective. 

Given two segments $AB$ and $ab$ we can interpolate between them by \say{gluing} the interpolations at proportional corresponding sides. In other words, if $\alpha$ and $\beta$ interpolate $A$ to $a$ and, respectively, $B$ to $b$, then
\begin{equation*}
    (1 - s)\alpha(t) + s\beta(t), 0\le s \le 1,
\end{equation*}
interpolates the segment from $\alpha(t)$ to $\beta(t)$. This leads to

\begin{definition}
Given two segments $AB$ and $ab$ the interpolation between them is the map
\begin{equation*}
    H(s, t) = (1 - s)(1 - t)A + (1 -s)ta + s(1-t)B + stb, 0\le t, s \le 1.
\end{equation*}
\end{definition}
\begin{remark}
   To define the interpolation we must beforehand chose which endpoint of $AB$ goes to $a$ and which one to $b$. We shall always tacitly assume they go in the corresponding order as they are written.
\end{remark}

By construction, at fixed level $0\le t_0 \le 1$, the set given by
\begin{equation*}
    H(s, t_0), 0\le s \le 1
\end{equation*}
is a line segment. At fixed level $s_0$ we get the parametrization
\begin{equation*}
    (1 - t)((1 - s)A + sB) + t((1 - s)a + sb)
\end{equation*}
This is the interpolation $Aa$ to $Bb$. However, the surface spanned by these interpolations itself is not necessarily contained in a plane. It is simply a doubly ruled surface (e.g. a saddle). We are speaking here in generic terms. Depending on the choice of points $A, a, B, b$ there can be self intersections and degenerations.

We are now ready to discuss the situation that we are interested in. We will have two flat quadrilaterals $ABCD$ and $abcd$. In all the cases we will encounter later these polygons will be rectangles and, in a degenerate case, one shall be a point. However, let us keep the discussion somewhat general to handle all cases together.

Let $H(s, t)$ and $G(s, t)$ be the interpolations $AB$ to $DC$ and $ab$ to $dc$. Then we can, once more, interpolate and get
\begin{equation*}
    F(s, t, u) = (1 - u)H(s, t) + uG(s, t), 0\le u \le 1.
\end{equation*}
Then $F$ is interpolating $ABCD$ towards $abcd$.

\begin{proposition}
    \label{prop: interpolation flat faces}
Let $0\le u_0 \le 1$ and define $A_{u_0}, B_{u_0}, C_{u_0}$ and $D_{u_0}$ be the points in the segments $Aa, Bb, Cc$ and $Dd$ dividing it in proportion $u_0$. Then the map $(s, t)\mapsto F(s, t, u_0)$ is the interpolation from $A_{u_0}B_{u_0}$ to $D_{u_0}C_{u_0}$.
\end{proposition}
\begin{proof}
  By definition
  \begin{align*}
      H(s, t) &= (1 - s)(1 - t)A + (1 -s)tD + s(1-t)B + stC,\\
      G(s, t) &= (1 - s)(1 - t)a + (1 -s)td + s(1-t)b + stc.
  \end{align*}
  Substituting and rearranging we get $F(s, t, u_0)$ equals
  \begin{equation*}
      (1 - s)(1 - t)((1-u_0)A + u_0a) + (1 -s)t((1-u_0)D + u_0d) + s(1-t)((1-u_0)B + u_0b) + st((1-u_0)C + u_0c).
  \end{equation*}
  The involved points 
  \begin{align*}
      A_{u_0} :=&(1-u_0)A + u_0a,\\
      D_{u_0} :=&(1-u_0)D + u_0d,\\
      B_{u_0} :=&(1-u_0)B + u_0b,\\
      C_{u_0} :=&(1-u_0)C + u_0c,
  \end{align*}
  are those at which the segments $Aa, Dd, Bb$ and $Cc$ get divided into proportion $u_0$. Furthermore, we recognize $F(s, t, u_0)$ as the interpolation from, $A_{u_0}B_{u_0}$ to $C_{u_0}D_{u_0}$. 
\end{proof}

\begin{definition}
    In the context of the above discussion, let us call the interpolation $(t, s)\mapsto F(s, t, u_0)$ the $u_0$-th level interpolation and its image the $u_0$-th level set of $F$.
\end{definition}

\begin{proposition}
\label{prop: hereditary conditions}
In the context of the above discussion, the following are hold:
\begin{enumerate}
    \item If $ABCD$ and $abcd$ are contained in parallel planes, then for all $0
    \le u \le 1$ the $u-th$ level sets is contained in a plane parallel to the planes of $ABCD$ and $abcd$.
    \item If $ABCD$ and $abcd$ are parallelograms, maybe not parallel among themselves, then for all $0
    \le u \le 1$ the $u-th$ level sets is a parallelogram.
\end{enumerate}
\end{proposition}
\begin{proof}
    We do them separately. 
    
\underline{Proof of 1:} There exists a vector $v$ perpendicular to $B-A, C-A, D-A$ and $b - a$, $c-a$ and $d-a$. This is because their corresponding planes are parallel. 

We have seen before that
\begin{equation*}
    A_{u} = (1-u)A + ua, B_u =(1-u)B + ub
\end{equation*}
and thus
\begin{equation*}
    B_u - A_u = (1 - u)(B - A) + u(b - a).
\end{equation*}
Hence, by linearity, $B_u - A_u$ is perpendicular to $v$. In the same, fashion $C_u - A_u$ and $D_u - A_u$ are perpendicular to $v$. Hence, we find that the four points are included in a plane through $A_u$ perpendicular to $v$. The result follows as we have seen the interpolation $A_uB_u$ to $D_uC_u$ is a ruled surface and the sides are in this plane.

\underline{Proof of 2:} 
We recall the reader that $XYWZ$ is a parallelogram if and only if
\begin{equation*}
    X + W = Y + Z.
\end{equation*}
By assumption then
\begin{equation*}
    A + C = B + D, a + c = b + d.
\end{equation*}
The result follows from the above equations by adding their appropriate linear combinations. Indeed,
\begin{equation*}
   A_u + C_u = (1 - u)(A + C) + u(a + c) = (1 - u)(B + D) + u(b + d) = B_u + D_u.
\end{equation*}
\end{proof}

\begin{definition}
    The image of $[0, 1]^3$ under the map $F$ will be called the prismoid induced by $F$. If no possibility of confusion can arise we will also call it the prismoid induced by the quadrilaterals $ABCD$ and $abcd$. 
\end{definition}
\begin{remark}
    We are slightly abusing of notation. A prismoid usually refers to a particular type of solid of which the above examples sometimes might not belong.
\end{remark}

Now, we must bring our attention to an important matter. When we have two segments $AB$ and $ab$ it is not true in general that the interpolation $H$ between them is injective or that the image of $H$ coincides with the interior of $ABba$ (if this is non intersecting). 

\begin{proposition}
\label{prop: convexplanarinjectivity}
    Let $AB$ and $ab$ be two coplanar segments such that $ABba$ is a convex quadrilateral. Then the interpolation $H$ between them is injective.
\end{proposition}
\begin{proof}
    Suppose $H(s_1, t_1) = H(s_2, t_2) = p$. Suppose that $s_1 > s_2$. We have said before that, at fixed level $s_1$, $H(s_1, \cdot)$ parametrizes a segment. This segments joins the corresponding points, at proportion $s_1$, of opposite sides. By convexity, the whole segment is included in $ABba$. In exactly the same way, there is a segment joining the corresponding points at proportion $s_2$ of the same opposite sides. The situation is thus as shown in the figure \ref{fig: impossible situation}. 

    \begin{figure}[H]
        \centering
        \includegraphics[scale = 0.2]{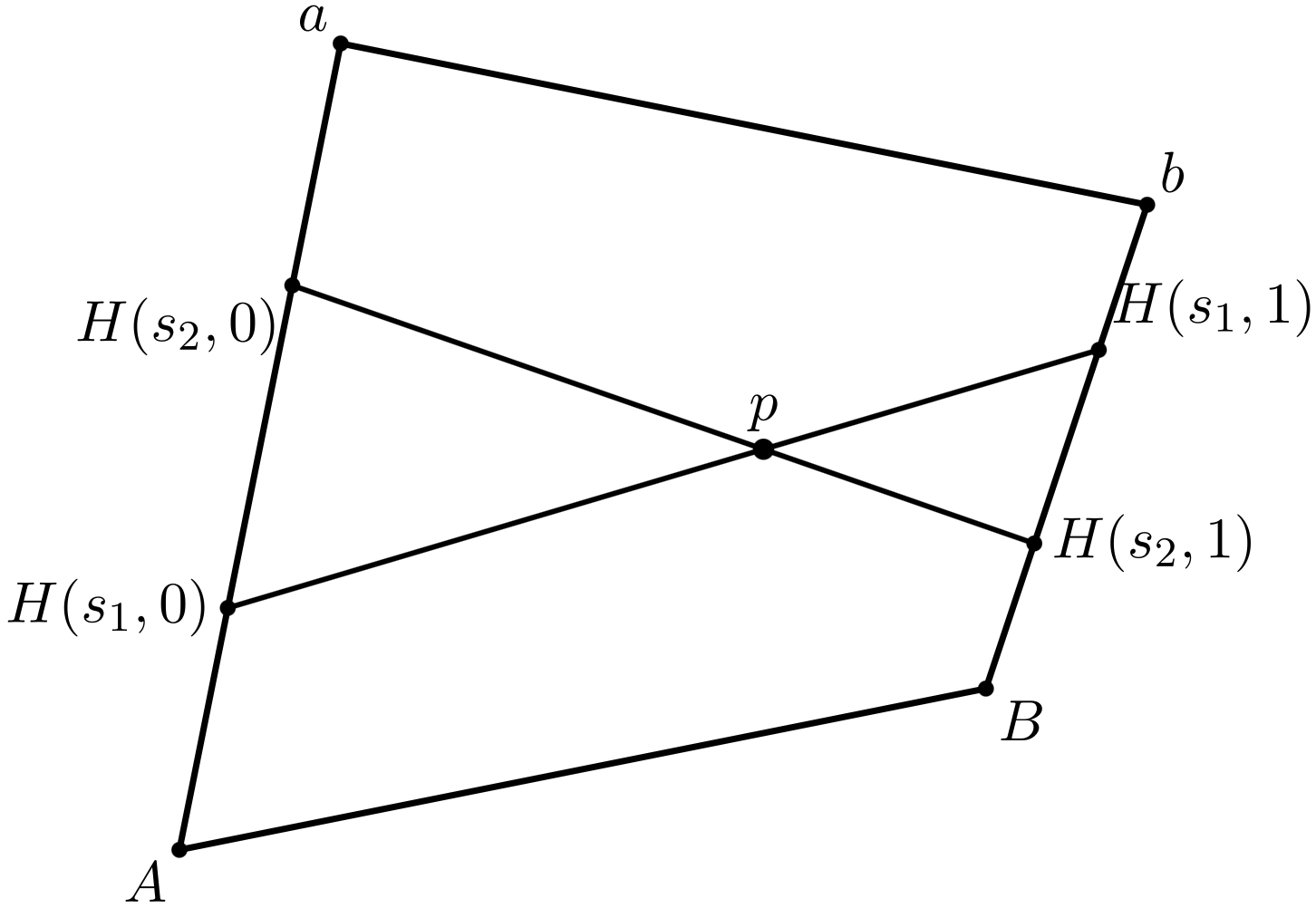}
        \caption{An impossible situation.}
        \label{fig: impossible situation}
    \end{figure}
    
    This configuration is impossible because on $Aa$ and $Bb$, the interpolations are injective. Consequently, as the proportion grows the points should be closer to the endpoints. However, what we see is that in one $s_1$ is closer while in the other it is $s_2$. This is a contradiction, thus $s_1 = s_2$.

    A similar argument shows $t_1 = t_2$.
\end{proof}

\begin{corollary}
   Let $AB$ and $ab$ be two coplanar segments such that $ABba$ is a convex quadrilateral. Then the interpolation $H$ between them is an homeomorphism onto $ABba$.
\end{corollary}
\begin{proof}
    The boundary of $[0, 1]^2$ goes to the sides of $ABba$. This is a simply connected curve and its interior is the quadrilateral (minus the edges). As $H$ is continuous, as it can be written as polynomials, we get its image is simply connected. Hence the map is surjective. We already know it is injective. We conclude $H$ is bijective. The closed map lemma implies $H$ is an homeomorphism onto its image.
\end{proof}

\begin{proposition}
    \label{prop: existence of projection}
    Let $ABCD$ and $abcd$ be two planar quadrilaterals. Let $F:[0, 1]^3\longrightarrow\mathbb{R}^3$ be the interpolation thus constructed from $ABCD$ and $abcd$. Let $\mathcal{S} = \text{Im}(F)$.
    Suppose that $F$ is injective. Then there exists a surjective continuous map
    \begin{equation*}
        \Phi: \mathcal{S}\longrightarrow [0, 1],
    \end{equation*}
    such that for each $0\le u \le 1$, the preimage $\Phi^{-1}(u)$ is the $u$-th level set of $F$.
\end{proposition}
\begin{proof}
    We emphasize that the content of this proposition is the existence and continuity of this map.
    For a point $p\in\mathcal{S}$ we define
    \begin{equation*}
        \Phi(p) = u,
    \end{equation*}
    if $p = F(s, t, u)$ for some point $(s, t, u)$. This is well-defined by the injectivity of $F$. It is clearly  surjective. By construction, the preimage of $u$ is the $u-$th level map of $F$ as this is defined as 
    \begin{equation*}
        F(s, t, u), 0\le s, t, \le 1.
    \end{equation*}
    Notice that we have the following diagram

   \begin{center}
    \begin{tikzcd}
        {[0, 1]^3} \ar[d, "F"] \ar[dr, "\Phi\circ F"] & \\
        \mathcal{S} \ar[r, "\Phi"] & {[0, 1]}
    \end{tikzcd}
\end{center}
The closed map lemma applied to $F$ implies that it is an homeomorphism onto its image (i.e. onto $\mathcal{S}$). Furthermore, 
    \begin{equation*}
        \Phi(F(s, t, u)) = u,
    \end{equation*}
    by definition. Hence, $\Phi = (\Phi\circ F)\circ F^{-1}$ is continuous. 
\end{proof}

\begin{corollary}
    Under the same conditions as those of proposition \ref{prop: existence of projection}, the $u$-th level sets of $F$ are topological disks.
\end{corollary}
\begin{proof}
    $F$ is an homeomorphism and the $u$-th level set is the image of $[0, 1]\times [0, 1] \times \{u\}$. 
\end{proof}

\begin{remark}
    There are configurations of $ABCD$ and $abcd$ where the injectivity of $F$ fails spectacularly. In the cases we will work with, the situation is so concrete that the injectivity is easily verified from all the arguments we have discussed above.
\end{remark}

\subsection{A case where injectivity holds.}

Now we describe a configurations we will find in our work and where injectivity holds.

\begin{proposition}
\label{prop: parallelograms injectivity}
    Let $ABCD$ and $abcd$ be parallelograms in parallel planes and with their corresponding vertices named in consistent cyclic order. Suppose they satisfy the following conditions:
    \begin{enumerate}
        \item $B - A$ and $b - a$ are positive scalar multiples of each other.
        \item $D - A$ and $d - a$ are positive scalar mulitples of each other.
    \end{enumerate}
    Then the interpolation $F$ is injective. 
\end{proposition}
\begin{proof}
    Let $0\le u \le 1$ be a fixed parameter. Define 
     \begin{align*}
      A_u :=&(1-u)A + ua,\\
      B_u :=&(1-u)B + ub,\\
      C_u :=&(1-u)C + uc,\\
      D_d :=&(1-u)D + ud,
  \end{align*}
  the points which cut the segments $Aa, Bb, Cc$ and $Dd$ at proportion $u$, respectively. 
  
  By assumption, we have the existence of two positive constants $\lambda, \mu >0$ such that
\begin{align*}
    B - A &= \lambda (b - a)\\
    D - A &= \mu (d - a).
\end{align*}
A simple computation then shows
\begin{align}
\label{eqn: scalar factor}
    B_u - A_u &= (\lambda + u(1 - \lambda))(b - a) = \left(1 - u +  \dfrac{u}{\lambda}\right)(B - A),\\
    \nonumber D_u - A_u &= (\mu + u(1 - \mu))(d - a) = \left(1 - u +  \dfrac{u}{\mu}\right)(D - A).
\end{align}
Notice that
\begin{align*}
    \lambda + u(1 - \lambda)&= \lambda\left(1 - u +  \dfrac{u}{\lambda}\right) > 0,\\
    \mu + u(1 - \mu) &= \mu\left(1 - u +  \dfrac{u}{\mu}\right) >0.
\end{align*}
This implies that the triangle $A_uB_uD_u$ has positive volume, as its sides are positive multiples of those of $ABD$. It also implies $A_u, B_u$ and $D_u$ has the same orientation as $A, B$ and $D$ (and $a, b$ and $d$). 

We have seen that, because the extreme faces are parallelograms, the intermediate $A_uB_uC_uD_u$ are also parallelograms. Nevertheless, we did not rule out they could be degenerate. However, the fact that $A_uB_uD_u$ has positive volume implies that it is nondegenerate because, being a parallelogram, its reflection $B_uD_uC_u$ must also have positive volume. We conclude that $A_uB_uC_uD_u$ is indeed a nondegenerate parallelogram. Hence, it is convex.

 Now, suppose for a contradiction, that
\begin{equation}
\label{eq: non injectivity equality}
    F(s_1, t_1, u_1) = F(s_2, t_2, u_2).
\end{equation}
We have seen before that the $u$-th level sets are planar and parallel to the quadrilaterals $ABCD$ and $abcd$. Their corresponding vertices cut the segments $Aa, Bb, Cc$ and $Dd$ in proportion $u$. 
Hence, if $u_1 \neq u_2$, these faces lie in different planes and the equality \eqref{eq: non injectivity equality} cannot hold. On the other hand, if $u_1 = u_2 = u_0$, then we have that
\begin{equation*}
    F(s_1, t_1, u_0) = F(s_2, t_2, u_0).
\end{equation*}
However, proposition \ref{prop: convexplanarinjectivity} implies that
\begin{equation*}
    (s, t) \mapsto F(s, t, u_0),
\end{equation*}
is injective. It is here where we needed convexity and nondegeneracy. Thus,
\begin{equation*}
    s_1 = s_2, t_1 = t_2.
\end{equation*}
Hence, the map $F$ is injective.
\end{proof}
\begin{remark}
    The case where the opposite faces are parallel rectangles is covered in this proposition. The two conditions are to guarantee that the labels of the rectangles are consistent.

    Furthermore, for rectangles we can say more. If $B - A$ and $D - A$ are orthogonal, then equations \eqref{eqn: scalar factor} imply that
    \begin{equation*}
        (B_u - A_u)\cdot (D_u - C_u) = 0.
    \end{equation*}
    Hence, all the level sets are also rectangles, as the only parallelograms with $90$ degree angles are rectangles.
\end{remark}

\subsection{Our main construction result}

We are ready to prove our main result in this appendix. Let us first state the result.

\begin{proposition}
\label{prop: appx PP/LP existance of phi}
   Let $X\subset Y$ be in a $PP$-configuration. Define
    \begin{equation*}
        \mathcal{S}_{XY} := \overline{Y - X},
    \end{equation*}
    that is, the closed set wrapped between the two polyhedrons with the surfaces involved.
    There exists a continuous map $\Phi: \mathcal{S}_{XY} \longrightarrow [0, 1]$ such that
    \begin{enumerate}
        \item For each $0\le u \le 1$, $\Phi^{-1}(u)$ is a rectangular prisms with faces parallel to those of $Y$. In particular, it is a polyhedral sphere $\mathbb{S}^2$.
        \item $\Phi^{-1}(0) = X$ and $\Phi^{-1}(1) = Y$.
    \end{enumerate}
    This results holds if $X$ is a point. The only change is the inequality in property $1$ should be $0 < u \le 1$.
\end{proposition}

We will do the nondegenarte case, that is, when $X$ is not a point. We leave the adaptation of the results to the degenerate case to the reader. 

We will have some preparation before we show the proof. Let us suppose we have two rectangular prisms $P_1$ and $P_2$ in a $PP$-configuration. We can divide $\mathcal{S}_{P_1, P_2}$ in different regions, each of which is covered by one of our previous constructions. We see this division in figure \ref{fig: All lofts of S}. We realize there are six prismoids in which we divide $\mathcal{S}_{P_1, P_2}$.
\begin{figure}[H]
    \centering
        \includegraphics[scale =0.4]{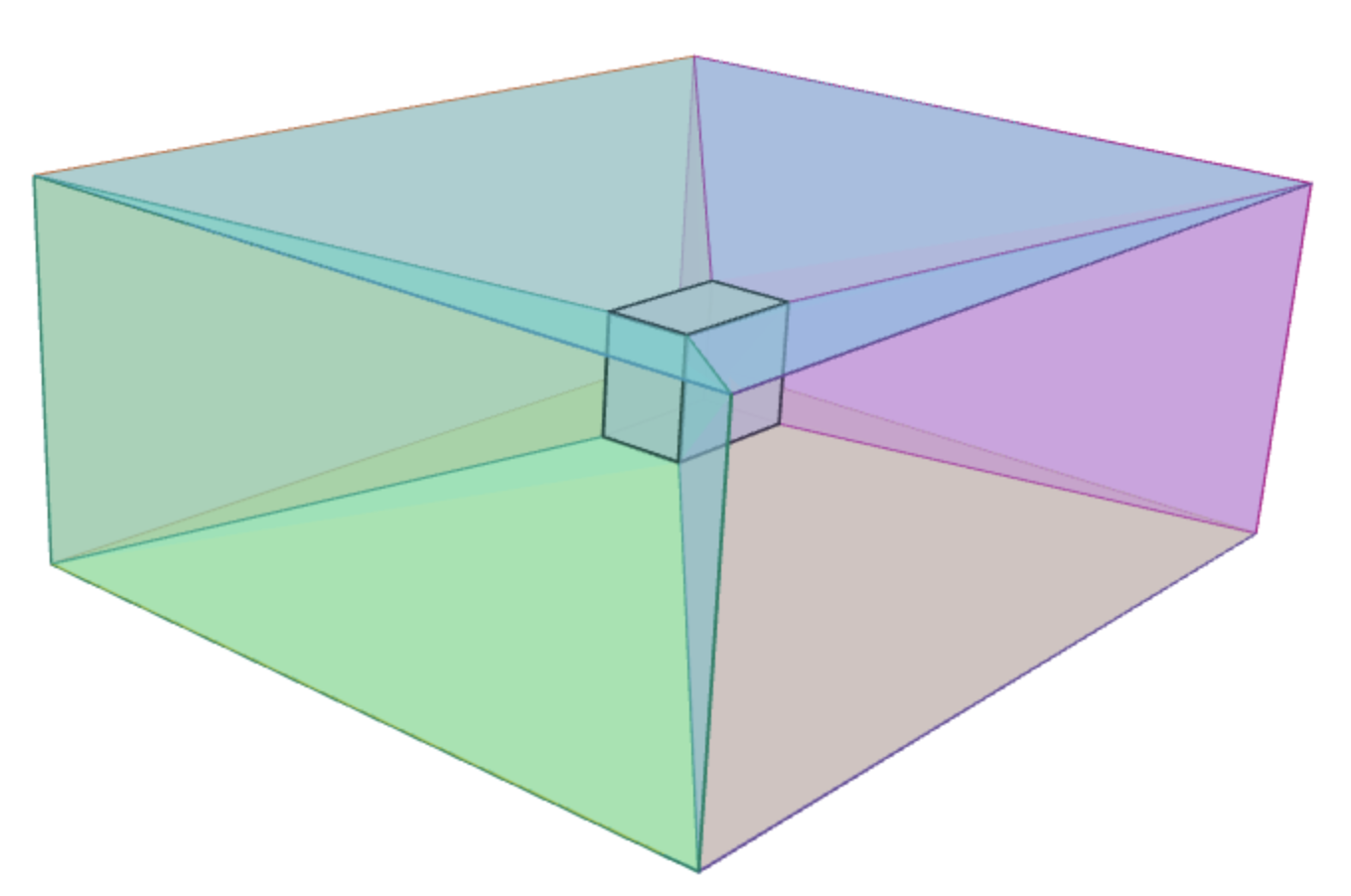}
    \caption{The different prismoids in which $\mathcal{S}_{P_1, P_2}$ is divided.}
    \label{fig: All lofts of S}
\end{figure}
For each of these prismoids $\mathcal{R}$, propositions \ref{prop: parallelograms injectivity} implies that the corresponding $F$ is injective. Consequently, proposition \ref{prop: existence of projection} implies the existence of a continuous surjective map
\begin{equation*}
    \Phi_{\mathcal{R}}: \mathcal{R}\longrightarrow [0, 1]
\end{equation*}
with $\Phi^{-1}(u)$ equal to its $u$-th level set. Our goal is to use the gluing lemma to build out of them a map
\begin{equation*}
    \Phi_{P_1, P_2}: \mathcal{S}_{P_1, P_2}\longrightarrow [0, 1].
\end{equation*}
We will then verify that this map satisfies the requirements we want. 

Our first order of business is to prove the hypothesis of the gluing lemma. That is, we must understand the intersections of the prismoids and prove that on them their corresponding maps coincide.

\begin{proposition}
    \label{prop: intersection case 1}
Suppose $ABCD, abcd, BCEH$ and $bceh$ are four planar convex quadrilaterals satisfying the following assumptions:
\begin{enumerate}
    \item $BCcb$ is a planar convex quadrilateral contained in the plane $L$.
    \item The points $A, D, a, d$ lie on different \textit{open} half space of $L$ than the points $E, H, e, h$. 
\end{enumerate}
Let $F$ be the interpolation from $ABCD$ to $abcd$ and $G$ the one from $BCEH$ to $bceh$. Then, the prismoids associated to $F$ and $G$ intersect exactly in their common face $BCcb$.
\end{proposition}
\begin{proof}
    Suppose that $F(s_1, t_1, u_1) = G(s_2, t_2, u_2)$. We have seen that the maps
    \begin{align*}
        u &\longmapsto F(s_1, t_1, u), \; 0\le u \le 1,\\
        u &\longmapsto G(s_2, t_2, u), \; 0\le u \le 1,
    \end{align*}
    are parametrization of segments. In other words, they are the interpolation from $F(s_1, t_1, 0)$ to $F(s_1, t_1, 1)$ and, respectively, $G(s_2, t_2, 0)$ to $G(s_2, t_2, 1)$.
    
    The first interpolation $F(s_1, t_1, u)$ has its endpoints on the quadrilaterals $ABCD$ and $abcd$. Because $ABCD$ and $abcd$ are convex quadrilaterals, and halfplane are convex themselves, these points are on the closed halplane that contains $ABCD$ and $abcd$. Consequently, the whole interpolation lies in this closed half plane. Analogously, the whole interpolation $G(s_2, t_2, u)$ lies on \textit{the other} closed halfplane. 
    
    The given situation is only possible if both segments lie on $L$ themselves. In particular, the intersection point lies on $L$. Furthermore, because $A, D, a, d$ are on the open half space (i.e. not on $L$), the only points of their corresponding prismoid on $L$ are the points of $BCcb$. Analogously for the other prismoid. We conclude, the intersection point lands on the face $BCcb$.

    On the other hand, because both prismoids share the vertices $B, C, c, b$, and this quadrilateral is convex, we know by construction the interpolation $BC$ to $bc$ spans the whole quadrilateral. 

    Thus, by double inclusion, the intersection of these prismoids is exactly their common face.
\end{proof}
\begin{remark}
    Notice that the injectivity of the prismoids is not necessary for this proof.
\end{remark}

All the cases we have to consider are covered by proposition \ref{prop: intersection case 1}. The opposite faces are two rectangles with parallel faces. The face their share is formed from an interpolation between two segments that are parallel.

\begin{proposition}
   \label{prop: the maps glue} 
The set of six maps $\Phi_{\mathcal{R}}:\mathcal{R}\longrightarrow [0, 1]$ glue to a map $\Phi:\mathcal{S}_{P_1, P_2}\longrightarrow [0, 1]$.
\end{proposition}
\begin{proof}
By propositions \ref{prop: intersection case 1}, we know that the six prismoids that comprise $\mathcal{S}_{P_1, P_2}$ intersect at their corresponding face. Each of this intersections is itself an interpolation from one segment to another. In particular, the definition of the map in these faces depends on the segments interpolated and not on the quadrilaterals to which they belong. 

We have already explained why each of the maps $\Phi_{\mathcal{R}}$ exists. Ultimately, it is because in the situations we are, we know the interpolation $F_{\mathcal{R}}$ is injective. Consequently, by construction, if a point $p$ in such prismoid is in the $u$-th level face, we have
\begin{equation*}
    \Phi_{\mathcal{R}}(p) = u.
\end{equation*}
Suppose two of these prismoids $\mathcal{R}_1$ and $\mathcal{R}_2$ intersect at the face $XYyx$, where $XY$ is one edge and $xy$ is the corresponding opposite edge. Then the intersection of the $u$-th level face, of either prismoid, with the shared face is the segment $X_uY_u$, where $X_u$ and $Y_u$ have their usual meaning (i.e. the points that divide $Xx$ and $Yy$ in proportion $u$). We know that every point on the shared face lies in a unique segment $X_uY_u$ and that the value of $u$ depends only $XYyx$, and not on the prismoid. We conclude, 
\begin{equation*}
    \Phi_{\mathcal{R}_1}(X_uY_u) = u = \Phi_{\mathcal{R}_2}(X_uY_u).
\end{equation*}
That is, $\mathcal{R}_1$ and $\mathcal{R}_2$ coincide on the shared face. This concludes our proof.
\end{proof}

We are finally ready to prove othe main result of this appendix.

\begin{proof}[Proof of proposition \ref{prop: appx PP/LP existance of phi}]
By proposition \ref{prop: the maps glue} we have the existence of a continuous map
\begin{equation*}
    \Phi:\mathcal{S}_{P_1, P_2}\longrightarrow [0, 1].
\end{equation*}
In each one of the prismoids $\mathcal{R}$ it coincides with $\Phi_{\mathcal{R}}$. To conclude, we must prove the properties hold. We do them separately.

\textit{Property 1: For each $0\le u \le 1$, $\Phi^{-1}(u)$ is a rectangular prism with faces parallel to those of $Y$. In particular, it is a polyhedral sphere $\mathbb{S}^2$.}

For each $0\le u \le 1$, the level set $\Phi^{-1}(u)$ consists of the union of the six level set faces $\Phi_{\mathcal{R}}^{-1}(u)$ as $\mathcal{R}$ varies on the ten prismoids. We know that for each case the level sets are rectangles, each of which glues on each of its edges to another one on the common face of the prismoids that contain them. We know by propositions \ref{prop: intersection case 1}, that the pairwise intersections, at the fixed level $t$, consists only on these edges. Consequently, we have a polyhedron with six faces, each of which is a rectangle.  

The only thing that we have to do is to compute its Euler Characteristic. There are six faces. The number of vertices is $8$, as there is one per edge joining corresponding vertices on $P_1$ and $P_2$. Finally, the number of edges are $12$. Indeed, there are four per face but each is counted twice. We have
\begin{equation*}
    \text{Vertices} - \text{Edges} + \text{Faces} = 8 - 12 + 6 = 2.
\end{equation*}
We conclude this is a sphere.

(What we are verifying is that there is no boundary. Since each face is moving at different pace, because the prisms are not homothetic, one might imagine there is a weird case where some faces fail to close.)

\textit{Property 2: $\Phi^{-1}(0) = P_1$ and $\Phi^{-1}(1) = P_2$.}

This follows from the fact that this holds for each interpolation. The beginning face, which is the $0$-th level face is a face of the $P_1$. The $1$-th level face is on $P_2$.
\end{proof}


\bibliographystyle{acm}
\bibliography{Bibliography}

\end{document}